\documentclass[12pt,letterpaper]{article}

\usepackage{renstyles}
\usepackage{cite}

\graphicspath{{./}{../}{../Figures/}{./Figures}}

\newcommand{\Sd}{{\bbS^{d-1}}}           
\newcommand{\nud}{\nu_d}                 
\newcommand{\omxy}{\omega_{\bx\by}}      
          
\newcommand{\Urom}{\widehat{U}}          
\newcommand{\Rb}{\cR}                    
\newcommand{\Keff}{K_{\mathrm{eff}}}     
\newcommand{\Kefff}{K_{\mathrm{eff}}'}  
\newcommand{\sgmsip}[2]{\langle #1,#2\rangle_{\sigma_s}}  
\newcommand{\sgmsnorm}[1]{\|#1\|_{\sigma_s}}

\title{Data-driven reduced-order models for the radiative transfer equation}

\author{
    Yinxi Pan\thanks{Department of Applied Physics and Applied Mathematics,
    Columbia University, New York, NY 10027; \href{mailto:yp2586@columbia.edu}{yp2586@columbia.edu}}
    \and
    Kui Ren\thanks{Department of Applied Physics and Applied Mathematics,
    Columbia University, New York, NY 10027; \href{mailto:kr2002@columbia.edu}{kr2002@columbia.edu}}
    \and
    Shanyin Tong\thanks{Department of Mathematics, University of Pennsylvania, Philadelphia, PA 19104; \href{mailto:tong3@sas.upenn.edu}{tong3@sas.upenn.edu}}
}

\begin{document}

\maketitle

\begin{abstract}

We present a data-driven reduced-order modeling (ROM) framework for the zeroth angular moment of the solution to the radiative transfer equation (RTE) rather than the full phase-space solution. Our construction is based on the Peierls integral formulation of the angularly averaged density. For media with isotropic scattering, the density satisfies a closed second-kind Fredholm equation with a globally attenuated, weakly singular kernel. We project this equation directly. For media with anisotropic scattering, we utilize the average-fluctuation decomposition to derive a closed system for a projection-based ROM. Numerical simulations are presented to illustrate the effectiveness of the ROMs we implemented. 

\end{abstract}

\begin{keywords}
radiative transfer equation, reduced-order model, reduced basis method, Peierls integral formulation, angular density, proper orthogonal decomposition
\end{keywords}

\begin{AMS}
    65N99, 65R20, 65D32, 85A25, 35Q20
\end{AMS}

\section{Introduction}
\label{SEC:Intro}

The radiative transfer equation (RTE) is a fundamental model for the transport of particles and photons through a scattering and absorbing medium, with applications ranging from astrophysics and neutron transport to biomedical optical imaging and remote sensing; see ~\cite{LeMi-Book93,
Mokhtar-Book97,CeBaBeAi-TTSP99,Arridge-IP99,Bal-IP09,Ren-CiCP10} and references therein. Let $\Omega\subset\bbR^d$ ($d=2,3$) be a bounded convex domain with smooth boundary $\partial\Omega$, and let $\Sd$ denote the unit sphere of admissible directions. We work on the phase space $X:=\Omega\times\Sd$ with incoming and outgoing boundary sets
\[
  \Gamma_{\pm} := \{(\bx,\bv)\in\partial\Omega\times\Sd \mid \pm\,\bv\cdot\bn(\bx)>0\},
\] 
$\bn(\bx)$ being the outward unit normal at $\bx\in\partial\Omega$. The stationary RTE with anisotropic scattering reads
\begin{equation}\label{EQ:RTE}
  \bv\cdot\nabla u + \sigma_t(\bx)\,u
  = \sigma_s(\bx)\int_{\Sd} p(\bv,\bv')\,u(\bx,\bv')\,d\bv' + f(\bx,\bv),
  \quad \text{in } X,
\end{equation}
subject to the standard no-incoming-flux boundary condition $u(\bx,\bv)=0$ on $\Gamma_-$.
Here $u(\bx,\bv)$ is the phase-space density of particles at $\bx$ traveling in direction $\bv$, $\sigma_t=\sigma_a+\sigma_s$ is the total attenuation coefficient, decomposed into absorption $\sigma_a>0$ and scattering $\sigma_s\ge0$ coefficients, $p(\bv,\bv')$ is the scattering phase function, and $f$ is an internal source. The angular measure $d\bv$ is normalized in the sense that $\int_{\Sd} d\bv = 1$ (in other words, $d\bv=d\omega/\nu_d$ where $d\omega$ is the Euclidean surface measure on $\Sd$ and $\nu_d:=|\Sd|$ denotes its unnormalized total area ($\nu_d=2\pi$ for $d=2$, $\nu_d=4\pi$ for $d=3$)). %
The phase function is normalized so that $\int_{\Sd}p(\bv,\bv')\,d\bv=1$ for every $\bv'$, i.e.\ scattering conserves particles. For well-posedness, we assume the standard sub-criticality and boundedness conditions: there exist constants $\gamma_0\in(0,1)$, $0<\sigma_0\le\sigma_1<\infty$ with
\begin{equation}\label{EQ:Assumption}
  \sup_{\bx\in\Omega}\frac{\sigma_s(\bx)}{\sigma_t(\bx)} = \gamma_0 < 1,
  \qquad \sigma_0 \le \sigma_t(\bx)\le\sigma_1 .
\end{equation}
A useful model for $p$ is the Henyey-Greenstein phase function, which depends on $\bv,\bv'$ only through the angle $\theta$ they form. Let $\cos\theta=\bv\cdot\bv'$, then
\begin{equation}\label{EQ:HG}
  p_{\rm HG}(\bv\cdot\bv') =
  \begin{dcases}
    \frac{1-g^2}{1+g^2-2g\cos\theta}, & d=2\\[0.10ex]
    \frac{1-g^2}{(1+g^2-2g\cos\theta)^{3/2}}, & d=3
  \end{dcases}
\end{equation}
with anisotropy parameter $g\in(-1,1)$. Isotropic scattering is when $g=0$ (thus $p\equiv1$).
 
The central difficulty in simulating~\eqref{EQ:RTE} is its dimensionality: the unknown lives on the $(2d-1)$-dimensional set $X$. Standard discretizations are accordingly expensive. This motivates a large literature on model reduction for transport, including but not limited to diffusion and asymptotic-preserving approximations, low-rank and dynamical low-rank methods, and projection-based ROMs~\cite{PeChChLi-JSC22,PeChChLi-MMS24,BuNaYa-JCP24,SaGoQiTi-arXiv24,Peng-JCP24,GaSi-NA25,FuTa-arXiv24}.

In this work, we develop a data-driven reduced-order model for the (rescaled) particle density of RTE, that is,
\begin{equation}\label{EQ:Density}
  U(\bx) := \aver{u}(\bx) := \int_{\Sd} u(\bx,\bv)\,d\bv
\end{equation}
This is simply the zeroth angular moment of the transport solution, which, with the normalized measure, differs from the physical fluence by the constant $\nud$. There are two main reasons to focus on $U(\bx)$ instead of $u(\bx,\bv)$. First, in many physical applications, $U$, not the full angular field, is the quantity of interest and the quantity actually measured. Second, the angular averaging ``smoothes" $u$ so that the manifold of the density $\{U(\cdot)\}$ is of lower dimension and thus more compressible than that of $u$. Therefore, reducing $U$ directly is both cheaper and better aligned with the application.

A naive projection-based ROM needs a governing operator acting on the very quantity being reduced, so that the reduced operator can be formed and the equation Galerkin-projected onto the reduced basis. In this work, we use the integral formulation of the transport equation~\eqref{EQ:RTE} to get a closed form equation for $U$. Integrating~\eqref{EQ:RTE} along
characteristics analytically eliminates the angular variable and yields a self-contained equation \emph{in physical space} for the density. In the
isotropic stationary case this is the classical Peierls integral equation~\cite{DaLi-Book93-6,ReZhZh-JCP19}: $U$ solves a second-kind Fredholm equation
\begin{equation}\label{EQ:RTE Integral Form}
  (\cI-K)\,U(\bx) = K(\sigma_s^{-1}f)(\bx), \qquad
  Kg(\bx) = \int_\Omega \cK(\bx,\by)\,g(\by)\,d\by,
\end{equation}
where the integral operator $K$ has a weakly singular, globally attenuated kernel
\begin{equation}\label{EQ:Kernel}
\cK(\bx,\by) = \sigma_s(\by)W(\bx,\by),\quad W(\bx,\by):=\dfrac{1}{\nud}\,
  \frac{E(\bx,\by)}{|\bx-\by|^{d-1}},
\end{equation}
with
\[
 E(\bx,\by) = \exp\!\Big(-\!\dint_0^{|\bx-\by|}\sigma_t\big(\bx-s\,\omxy\big)\,ds\Big)\,,
\]
being the line-of-sight attenuation and $\omxy=(\bx-\by)/|\bx-\by|$. This equation is exact: it holds well beyond the diffusion limit and, crucially for ROM, its kernel is a globally supported, smoothly attenuated kernel whose discretized operator has rapidly decaying singular values, exactly the structure that makes the projection-based reduction effective. In the anisotropic case, similar integral formulations are developed and used to formulate the corresponding ROMs. Besides the proposed ROM algorithms, we also provide the theoretical analysis of their stability and accuracy, which are further illustrated by several numerical experiments in either isotropic or anisotropic settings.

Model reduction for kinetic and transport equations has grown into a broad field. We briefly survey the related literature in a few directions, and discuss their differences from our work. 

Reduced-basis and POD-Galerkin methods have been built directly on the transport equation: greedy reduced-basis solvers for the (parametric, steady, and time-dependent) radiative transfer equation~\cite{PeChChLi-JSC22,PeChChLi-MMS24,
MaChChLi-JCP25}, POD reduction of the \emph{angular} variable for neutron/photon transport and for non-grey media~\cite{BuCaGoDaFaPaNa-JCP15,HuBu-IJNME22,SoBuDaPa-JQSRT19}, space-angle and space-time reduced
models~\cite{BuNaYa-JCP24,ChBrArAnHu-JCP20}, and ROM-accelerated iterative solvers that embed a reduced model as a synthetic-acceleration preconditioner or offline/online forward
map~\cite{Peng-JCP24,FuTa-arXiv24,EtFaYi-JCP23}. Most of these reduce the full phase-space field $u$ (often in the angular variable). We instead reduce the \emph{density} $U$ on its closed integral equation, so the reduced unknown lives in physical space and inherits the smoothing and contraction of the integral operator.
 
A complementary line reduces a macroscopic quantity by closing a moment hierarchy with data-driven approximations of the Eddington tensor / variable Eddington factor~\cite{CoAn-JQSRT23,CoAn-JCP24} or with structure-preserving and entropy-based machine-learning closures~\cite{HuChChRo-JCP22,HuEtAl-MMS24,ScLaFrHa-JCP25}. Neural-operator surrogates (DeepONet/FNO) and related neural-network surrogates learn the parameter-to-solution map directly~\cite{SaChPaJi-arXiv26,KoKrReSc-MN24}. Our density $U$ is the zeroth such moment, but rather than \emph{learning} a closure we use the \emph{exact} closure furnished by the integral formulation (the isotropic Peierls equation, and the Schur-complement closure of \Cref{THM:ANISO Closed} for anisotropy).
 
A large body of work compresses the phase-space solution \emph{in situ} by dynamical low-rank approximation (DLRA)~\cite{KoLu-SIMAX07,DiEiLi-SINUM21,EiHuYi-SISC21},
including asymptotic-preserving, energy-stable, conservative, and rank-adaptive integrators for (thermal) radiative transfer and the Vlasov equation~\cite{EiHuKu-SINUM24,PaFrKu-arXiv24,PaKu-arXiv25,BaEiKlKu-arXiv25,
CeKuLu-SISC24,EiJo-JCP21,GuQi-JSC24,UsZe-arXiv25}, as well as hierarchical-Tucker and tensor-train representations~\cite{SaGoQiTi-arXiv24,Ju-arXiv26}. Our reduction is instead offline/online and projection-based on the angularly-averaged integral equation.

Dynamic mode decomposition accelerates transport sweeps and computes time eigenvalues~\cite{Mc-NSE19,McHa-JCP21}, and related reductions exist for plasma (Vlasov) kinetics~\cite{TyKr-CPP23,TsCgGhLoChBe-arXiv23,PaVi-JCP25}, the nonlinear Boltzmann equation of rarefied gases~\cite{CaLiLuWa-arXiv25,Sa-JCP21,SuZo-arXiv25}, and the phonon Boltzmann transport
equation~\cite{HoGrRo-APL23,VaRoPe-arXiv25}. 

The rest of the paper is organized as follows. We first review the integral formulations and their analytic properties in \Cref{SEC:Integral Form}. We then develop the density-based ROM in~\Cref{SEC:ROM}, along with its projection variants, error analysis, and offline/online structure. \Cref{SEC:FOM} describes the full-order discretization of the models. Numerical simulations are provided in~\Cref{SEC:Num} to validate the proposed algorithms. Concluding remarks are offered in~\Cref{SEC:Concl}. 

\section{Integral formulations of RTE}
\label{SEC:Integral Form}

Let us recall the classical integral formulation, following the
notation of~\cite{ReZhZh-JCP19,ReZhZh-SIAM21,FaAnYi-JCP19,ZhZh-CSIAM20,
EtFaYi-JCP23}.

\subsection{The isotropic scattering case}
\label{SUBSEC:Integral ISO}

When the scattering is isotropic, that is, $p(\bv,\bv')\equiv 1$, the RTE simplifies to
\begin{equation}\label{EQ:RTE ISO}
\begin{aligned}
  \bv\cdot\nabla u(\bx,\bv) + \sigma_t(\bx)\,u(\bx,\bv)
  &= \sigma_s(\bx)\,U(\bx) + f(\bx), && \text{in } X,\\
  u(\bx,\bv) &= 0, && \text{on } \Gamma_-,
\end{aligned}
\end{equation}
where $U=\aver{u}$ as in~\eqref{EQ:Density} and we assumed an isotropic source $f=f(\bx)$. Let us write the right-hand side as the (angularly isotropic) source $S(\bx):=\sigma_s(\bx)U(\bx)+f(\bx)$. Freezing $S$, integrating the linear transport operator along the characteristic through $\bx$ in direction $\bv$, and using the boundary condition, gives
\begin{equation}\label{EQ:Characteristic Sol}
  u(\bx,\bv) = \int_0^{\tau_-(\bx,\bv)}
  \exp\!\Big(-\!\int_0^\ell \sigma_t(\bx-s\bv)\,ds\Big)\,
  S(\bx-\ell\bv)\,d\ell,
\end{equation}
where the backward exit distance is
$\tau_-(\bx,\bv):=\sup\{s\mid \bx-s'\bv\in\Omega \ \forall\, 0\le s'<s\}$.
Averaging~\eqref{EQ:Characteristic Sol} over $\bv\in\Sd$ and converting the
$(\ell,\bv)$ integral to a volume integral over $\by=\bx-\ell\bv$ through the polar
change of variables $\ell=|\bx-\by|$, $\bv=\omxy=(\bx-\by)/|\bx-\by|$,
\begin{equation}\label{EQ:Polar}
  d\by = \nud\, \ell^{d-1}\,d\ell\,d\bv ,
\end{equation}
yields the closed second-kind Fredholm equation for the density $U$ given in~\eqref{EQ:RTE Integral Form}, whose right-hand side we abbreviate as $\phi:=K(\sigma_s^{-1}f)$. The function $E$ is symmetric in the sense that $E(\bx,\by)=E(\by,\bx)$ (the integral is taken over the same segment). When $\sigma_s$ and $\sigma_t$ are constant, the kernel reduces to $\cK(\bx,\by)=\frac{\sigma_s}{\nu_d}\,
e^{-\sigma_t|\bx-\by|}/|\bx-\by|^{d-1}$. 

\begin{remark}
Let us emphasize that equation~\eqref{EQ:RTE Integral Form} is \emph{exact}: no diffusion or small-mean-free-path approximation is made. It reduces the unknown from the $(2d-1)$-dimensional phase space to the $d$-dimensional physical space and may be solved by fast integral-equation methods (e.g.\ FMM-accelerated solvers~\cite{ReZhZh-JCP19,FaAnYi-JCP19,ReZhZh-SIAM21}). It is the exactness across regimes, together with the kernel's smoothing, that we exploit for model reduction.
\end{remark}

It is a classical result~\cite{Agoshkov-Book12,DaLi-Book93-6} that $K$ is a contraction in the following sense.
\begin{proposition}[\cite{Agoshkov-Book12,DaLi-Book93-6}]\label{PROP:Contraction}
Under~\eqref{EQ:Assumption}, assume further that $\sigma_s(\bx)\ge \sigma_{s,\min}$ in $\Omega$ for some $\sigma_{s,\min}>0$. Then the operator $K$ in~\eqref{EQ:RTE Integral Form} is a contraction on $L^\infty(\Omega)$ with
\begin{equation}\label{EQ:Contraction}
  \|K\|_{L^\infty\to L^\infty}
  = \sup_{\bx\in\Omega}\int_\Omega \cK(\bx,\by)\,d\by
  \le \sup_{\bx\in\Omega}\frac{\sigma_s(\bx)}{\sigma_t(\bx)} =: \gamma_0 < 1,
\end{equation}
hence $\cI-K$ is boundedly invertible and~\eqref{EQ:RTE Integral Form} has a unique solution $U=(\cI-K)^{-1}\phi\in L^\infty(\Omega)$.
\end{proposition}

Proposition~\ref{PROP:Contraction} is the analytic backbone of the ROM error analysis in~\Cref{SUBSEC:Error}: because $\|K\|<1$, both the full and the reduced second-kind problems are uniformly well-posed, and the reduced residual controls the reduced error.

\subsection{Fundamental properties given by integral formulation}
\label{SUB:structure}

Our main motivation, that is, the density manifold is generally more compressible than the phase-space manifold, can be made more precise. We now recall a few properties of the integral operator $K$ that together explain why the density-based ROM works. 
\medskip

\noindent{\bf Smoothing property.} The angular averaging that produced~\eqref{EQ:RTE Integral Form} converts the transport operator into a compact, smoothing integral operator. Indeed, from~\eqref{EQ:Kernel} and $E\le1$, $0\le\sigma_s\le\sigma_1$, we have
\begin{equation}\label{EQ:Weak Sing}
  0\le \cK(\bx,\by)\le \frac{\sigma_1}{\nud}\,\frac{1}{|\bx-\by|^{d-1}}\,.
\end{equation}
Therefore, $\cK$ is weakly singular (the exponent $d-1$ is strictly below the space dimension $d$, hence locally integrable).

\begin{lemma}\label{LEM:Smoothing}
Let $\sigma_t,\sigma_s\in L^\infty(\Omega)$ satisfy~\eqref{EQ:Assumption}. Then $K$ is a compact operator on both $L^2(\Omega)$ and $C(\overline\Omega)$, and it smooths: it maps $L^\infty(\Omega)$ continuously into the H\"{o}lder space $C^{0,\lambda}(\overline\Omega)$ for every $\lambda\in(0,1)$, with $\|Kg\|_{C^{0,\lambda}}\le C_\lambda\|g\|_{L^\infty}$. Consequently the density $U=(\cI-K)^{-1}\phi=\phi+KU$ is H\"{o}lder continuous, $U\in C^{0,\lambda}(\overline\Omega)$, for every source $f\in
L^\infty(\Omega)$.
\end{lemma}
\begin{proof}
The bound~\eqref{EQ:Weak Sing} is the defining estimate of a weakly singular
kernel with singularity exponent $1$ on a $d$-dimensional domain. Compactness on
$C(\overline\Omega)$ and $L^2(\Omega)$ and the mapping into $C^{0,\lambda}$,
$\lambda<1$, are then classical~\cite[Thms.~2.22--2.30]{Kress-Book14}. The Lipschitz dependence of the attenuation $E(\bx,\by)$ on its endpoints (a consequence of $\sigma_t\in L^\infty$) does not affect the singularity order. The last claim follows from $U=\phi+KU$ with $\phi,KU\in C^{0,\lambda}$.
\end{proof}

In general, the phase-space solution operator $f\mapsto u$ has no analogous smoothing away from the diffusion limit, so $\{u^f\}$ resists linear compression, whereas Lemma~\ref{LEM:Smoothing} confines $\{U^f:\|f\|_\infty\le1\}$ to a bounded subset of $C^{0,\lambda}$, which is
precompact in $C(\overline\Omega)$ by Arzel\`a--Ascoli. Precompactness is
exactly the qualitative statement that the density manifold can be approximated to any tolerance by a finite-dimensional space. Proposition~\ref{PROP:n-Width} below makes the rate quantitative.
\medskip

\noindent{\bf Hidden self-adjointness.} The kernel factors as $\cK(\bx,\by)=\sigma_s(\by)\,W(\bx,\by)$ with the symmetric attenuation weight $W(\bx,\by)=E(\bx,\by)/(\nud|\bx-\by|^{d-1})=W(\by,\bx)$ (cf.~\eqref{EQ:Kernel}). This algebraic structure makes $K$ self-adjoint in a weighted inner product, a fact we exploit to sharpen the ROM error constant and to choose the POD inner product.

\begin{lemma}\label{LEM:Symm}
Under the assumption in~\eqref{EQ:Assumption} and the additional lower bound $\sigma_s\ge \sigma_{s,\min}>0$, we equip $L^2(\Omega)$ with the $\sigma_s$-weighted inner product
\[
    \sgmsip{f}{g}:=\int_\Omega \sigma_s(\bx)\,f(\bx)g(\bx)\,d\bx\,,
\]
and the corresponding norm $\|f\|_{\sigma_s}:=\sgmsip{f}{f}^{1/2}$. Then $K$ is self-adjoint with respect to $\sgmsip{\cdot}{\cdot}$. The spectrum of $K$ is
real, and
\begin{equation}\label{EQ:Real Spec}
  \operatorname{spec}(K)\subset[-\gamma_0,\,\gamma_0]\subset(-1,1),
  \qquad \|K\|_{\sigma_s} = \rho(K)\le \gamma_0 .
\end{equation}
In particular, $\cI-K$ is self-adjoint and positive-definite in $\sgmsip{\cdot}{\cdot}$, with $\sgmsip{(\cI-K)g}{g}\ge(1-\gamma_0)\sgmsnorm{g}^2$.
\end{lemma}
\begin{proof}
The self-adjointness of $K$ in $\sgmsip{\cdot}{\cdot}$ follows from the identity
$\sigma_s(\bx)\cK(\bx,\by)=\sigma_s(\by)\cK(\by,\bx)$, which holds because both sides equal $\sigma_s(\bx)\sigma_s(\by)W(\bx,\by)$ and $W$ is symmetric. Since $K$ is compact and self-adjoint in $\sgmsip{\cdot}{\cdot}$, its spectrum is real and $\|K\|_{\sigma_s}=\rho(K)$. Meanwhile, every nonzero eigenvalue of $K$ is shared with $L^\infty(\Omega)$: an eigenfunction $\phi=\lambda^{-1}K\phi$ with $\lambda\neq0$ is H\"older continuous by the smoothing of Lemma~\ref{LEM:Smoothing} (after a finite integrability bootstrap, as $K$ is a weakly singular operator that raises integrability), hence $\phi\in L^\infty(\Omega)$ and $\lambda$ is an eigenvalue of $K$ on $L^\infty$ as well. Thus $|\lambda|\le\|K\|_{L^\infty\to L^\infty}\le\gamma_0$ by~\eqref{EQ:Contraction}, so that $\operatorname{spec}(K)\subset[-\gamma_0,\gamma_0]$ and $\|K\|_{\sigma_s}=\rho(K)\le\gamma_0$. Positivity of $\cI-K$ follows from $\operatorname{spec}(\cI-K)\subset[1-\gamma_0,1+\gamma_0]$.
\end{proof}

The above result will be used in~\Cref{SEC:ROM}, where computing the
POD in the $\sigma_s$-weighted inner product makes the Galerkin reduced operator $\widehat\bK$ symmetric and renders the reduced problem the energy-orthogonal projection of the full one, which improves the C\'ea constant of Theorem~
\ref{THM:Cea} from $\tfrac{1+\kappa}{1-\kappa}$ to its square root
(Corollary~\ref{COR:ROM Energy}).  

\noindent{\bf Kolmogorov-width bound.} The error analysis of \Cref{SEC:ROM} bounds the ROM error by the best-approximation error of $U$ in the reduced space. The relevant intrinsic quantity is therefore the Kolmogorov $n$-width of the solution manifold $\cM:=\{U^{f(\mu)}:\mu\in\cP\}$ over the admissible parameter set $\cP$,
\[
    d_n(\cM):=\inf_{\dim V_n=n}\sup_{U\in\cM}\inf_{v\in V_n}\|U-v\|\,.
\]
The following result is a specialization to the density map of the polynomial-approximation
theory for holomorphic parametric problems~\cite[Sections 3-4]{CoDe-AN15}. 
\begin{proposition}
\label{PROP:n-Width}
Suppose the source $f(\cdot;\mu)$ depends holomorphically on a parameter $\mu$ ranging over a compact set $\cP\subset\bbR^p$ that admits a complex neighborhood on which $\mu\mapsto f(\cdot;\mu)\in L^\infty(\Omega)$ is bounded and holomorphic. Then the map
$\mu\mapsto U^{f(\mu)}=(\cI-K)^{-1}K(\sigma_s^{-1}f(\mu))$ is holomorphic into $C^{0,\lambda}(\overline\Omega)$, and the Kolmogorov widths of $\cM$ decay sub-exponentially in the sense that
\begin{equation}\label{EQ:n-Width}
  d_n(\cM)\ \le\ C\,\exp\!\big(-c\,n^{1/p}\big),
\end{equation}
for constants $C,c>0$ depending on $\gamma_0$, $\sigma_1$, and the domain of holomorphy. 
\end{proposition}
\begin{proof}
By Proposition~\ref{PROP:Contraction}, $(\cI-K)^{-1}$ is a bounded operator on $L^\infty$. By Lemma~\ref{LEM:Smoothing}, $K(\sigma_s^{-1}\cdot)$ maps into $C^{0,\lambda}$. Composing with the holomorphic source map shows that $\mu\mapsto U^{f(\mu)}$ is
holomorphic with values in $C^{0,\lambda}$. Holomorphic dependence on $p$ parameters yields the $n$-width bound~\eqref{EQ:n-Width} by the polynomial-approximation
theory of parametric operator
equations~\cite[Sections 3-4]{CoDe-AN15}. 
\end{proof} 
 
In the diffusive scaling $\sigma_t=\bar\sigma_t/\epsilon$ and $\sigma_a=\epsilon\,\bar\sigma_a$ ($\epsilon\to 0$, so $\gamma_0\to1$), the kernel $\cK$ concentrates near the diagonal and $(\cI-K)^{-1}$ converges, as $\epsilon\to0$, to the solution operator of a diffusion equation $-\nabla\!\cdot\!\big(\tfrac{1}{d\,\bar\sigma_t}\nabla U\big)+\bar\sigma_a U=f$, whose solutions are $H^2$-smooth~\cite{DaLi-Book93-6}. The density manifold is then
maximally compressible. The smoothing of Lemma~\ref{LEM:Smoothing} is not tied to this limit. It holds for every $\gamma_0<1$. So reducibility persists in the transport and intermediate regimes where no diffusion approximation is available. 

\subsection{The anisotropic scattering case}
\label{SUBSEC:ANISO}

In the anisotropic scattering case, the density $U$ does not satisfy a self-contained second-kind equation with the simple kernel~\eqref{EQ:Kernel}. It couples to the angular structure of the
in-scattering source. We use the standard average-fluctuation decomposition to derive a closed-form integral equation for $U$.

The central observation is that the $U$-$u$ coupling can be eliminated exactly, yielding a closed equation for $U$ whose anisotropy correction is expressed through an auxiliary, angularly resolved operator. 

Let us define the in-scattering source and the total source for \eqref{EQ:RTE}
\begin{equation*}
  \psi(\bx,\bv) := \int_{\Sd} p(\bv,\bv')\,u(\bx,\bv')\,d\bv',
  \qquad q(\bx,\bv) : = \sigma_s(\bx)\,\psi(\bx,\bv) + f(\bx,\bv),
\end{equation*}
so the stationary anisotropic RTE \eqref{EQ:RTE} reads $\bv\cdot\nabla u+\sigma_t u = q$. Integrating along characteristics as in~\eqref{EQ:Characteristic Sol} with the direction-dependent source $q$, and averaging over $\bv$ with the change of variables~\eqref{EQ:Polar}, gives the identity
\begin{equation*}
  U(\bx) = \int_\Omega W(\bx,\by)\,q\big(\by,\omxy\big)\,d\by,
  \qquad \omxy=\frac{\bx-\by}{|\bx-\by|}\,,
\end{equation*}
where $W(\bx, \by)$ is defined in~\eqref{EQ:Kernel}. Let $\psi_\perp(\bx, \bv)$ be a zero-mean fluctuation such that 
\begin{equation}\label{EQ:psi-Split}
  \psi(\bx,\bv) = U(\bx) + \psi_\perp(\bx,\bv),
  \qquad \int_{\Sd}\psi_\perp(\bx,\bv)\,d\bv = 0.
\end{equation}
Introducing the \emph{anisotropic part} of the phase function
\begin{equation}\label{EQ:a-Def}
  a(\bv,\bv') := p(\bv,\bv') - 1,
  \qquad \int_{\Sd} a(\bv,\bv')\,d\bv = 0,
\end{equation}
and using $\psi=\int (a+1)u\,d\bv' = \int a\,u\,d\bv' + U$, the fluctuation is
precisely the $a$-weighted average of $u$,
\begin{equation}\label{EQ:psiperp-Def}
  \psi_\perp(\bx,\bv) = \int_{\Sd} a(\bv,\bv')\,u(\bx,\bv')\,d\bv'.
\end{equation}
In the isotropic limit $g=0$ of the HG phase function \eqref{EQ:HG}, we have $a\equiv0$, hence $\psi_\perp\equiv0$ and
$\psi\equiv U$, recovering ~\eqref{EQ:RTE Integral Form}.

The split~\eqref{EQ:psi-Split} into an angular mean and a mean-zero fluctuation is the \emph{micro--macro} (equivalently, moment, or even-/odd-parity) decomposition of kinetic theory: $U$ is the projection of the in-scattering source onto the null space of the isotropic scattering operator (the angular average), and $\psi_\perp$ is its orthogonal, mean-zero complement, carrying precisely the higher angular moments of $u$ (cf.~\eqref{EQ:psiperp-Def}). The companion split of the phase function $p=1+a$ with $\int_{\Sd} a\,d\bv=0$ is the classical separation of the isotropic ($\ell=0$) Legendre/spherical-harmonic mode~\cite{LeMi-Book93}. The same decomposition underlies asymptotic-preserving schemes~\cite{LemMie-SISC08} and reduced-basis models~\cite{PeChChLi-MMS24} for kinetic transport, and its mean part is the leading-order (diffusion) term in the small-mean-free-path expansion~\cite{LaKe-JMP74,bss1}. What is specific to our development is not the decomposition but its \emph{exact} closure below: eliminating $\psi_\perp$ by a Schur complement on the integral formulation yields a self-contained second-kind equation for the density $U$ alone; see~\Cref{THM:ANISO Closed}.

Let us define the angularly dependent kernel
\[
  \wt\cK(\bx,\by;\bv) :=\sigma_s(\by)\,W(\bx,\by)\,a\!\left(\bv,\omxy\right)\,,
\]
and two integral operators $H$ and $G$ acting on the fluctuation field
\begin{equation}\label{EQ:HG-ops}
\begin{array}{rcl}
(H h)(\bx,\bv) &:=& \dint_\Omega \wt\cK(\bx,\by;\bv)\,h\!\left(\by,\omxy\right)\,d\by,\\[2ex] 
(G g)(\bx,\bv) &:=& \dint_\Omega \wt\cK(\bx,\by;\bv)\,g(\by)\,d\by\,.
\end{array}
\end{equation}
$H$ and $G$ act on angularly resolved and angularly independent fields, respectively. Because $\int_{\Sd}a(\bv,\omxy)\,d\bv=0$ for every fixed $\omxy$, both operators produce \emph{zero-mean} fields,
\begin{equation}\label{EQ:HG-zeromean}
  \int_{\Sd}(Hh)(\bx,\bv)\,d\bv = 0,
  \qquad \int_{\Sd}(Gg)(\bx,\bv)\,d\bv = 0\,.
\end{equation}
Inserting the characteristic representation $u(\bx,\bv)=\int_0^{\tau_-}I(s)\,q(\bx-s\bv,\bv)\,ds$ (with attenuation $I(s):=\exp(-\int_0^s\sigma_t(\bx-s'\bv)\,ds')$) into~\eqref{EQ:psiperp-Def}, using the decomposition $q=\sigma_s(U+\psi_\perp)+f$ and the change of variables~\eqref{EQ:Polar}, the fluctuation satisfies its own second-kind equation,
\begin{equation}\label{EQ:psiperp-Eq}
  (\cI - H)\,\psi_\perp = G\,U + H\big(\sigma_s^{-1}f\big)\,.
\end{equation}
Averaging the same representation (now keeping the mean part) reproduces,
through the isotropic operator $K$ of~\eqref{EQ:RTE Integral Form} acting on fields evaluated at $\omxy$,
\begin{equation}\label{EQ:U Eq}
  (\cI - K)\,U = K\,\psi_\perp + K\big(\sigma_s^{-1}f\big).
\end{equation}
Equations~\eqref{EQ:psiperp-Eq}--\eqref{EQ:U Eq} are an \emph{exact}, closed
coupled system for the pair $(U,\psi_\perp)$: a spatial scalar density coupled to an angularly resolved, zero-mean fluctuation.

The following statement is elementary. It assembles two standard facts, the contractivity of the sub-critical multiple-scattering operator~\cite{DaLi-Book93-6,Agoshkov-Book12} and the fact that an orthogonal compression cannot increase the operator norm~\cite{HoJo-Book85}, to prove the invertibility of $\cI-H$ on the space of zero-mean fluctuations. 

\begin{lemma}\label{LEM:Hcontraction}
Let the phase function be nonnegative and rotationally invariant, $p=p(\bv\cdot\bv')\ge0$ (so that $\int_{\Sd}p(\bv,\bv')\,d\bv=\int_{\Sd}p(\bv,\bv')\,d\bv'=1$). Then, under the sub-criticality assumption~\eqref{EQ:Assumption}, the operator $\cI-H$ is boundedly invertible on the zero-mean subspace
\[
V_\perp=\{\psi\in L^2(\Omega\times\Sd):\int_{\Sd}\psi\,d\bv=0\}\,
\]
with the properties
\begin{equation}\label{EQ:Hcontraction}
  \|H\|\le \sqrt{\gamma_0}<1,\qquad \|(\cI-H)^{-1}\|\le(1-\sqrt{\gamma_0})^{-1},
\end{equation}
in the weighted Hilbert norm in which the scattering operator is self-adjoint.
\end{lemma}
\begin{proof}
We split $L^2(\Omega\times\Sd)=V_0\oplus V_\perp$ into the angularly constant fields $V_0\cong L^2(\Omega)$ (carrying $U$) and their zero-mean orthogonal complement $V_\perp$ (carrying $\psi_\perp$). With respect to this split, the multiple-scattering operator, whose Neumann series solves the transport problem, is the block operator
\[
  \cS=\begin{pmatrix}K & D\\[2pt] G & H\end{pmatrix},
  \qquad D:\psi_\perp\mapsto K\psi_\perp\ \text{of}~\eqref{EQ:U Eq}\,.
\]
So the coupled system~\eqref{EQ:psiperp-Eq}-\eqref{EQ:U Eq} is
$(\cI-\cS)(U,\psi_\perp)^{\!\top}=\big(K(\sigma_s^{-1}f),H(\sigma_s^{-1}f)\big)^{\!\top}$, and $H=P_\perp\,\cS\,P_\perp$ is the orthogonal compression of $\cS$ to $V_\perp$ ($P_\perp$ the orthogonal angular-mean projection). It therefore suffices to show $\|\cS\|<1$.

We factorize $\cS=P_\Sigma\,T\,M_{\sigma_s}$, where $M_{\sigma_s}\psi=\sigma_s\psi$, the transport solve $T:g\mapsto u$ inverts $\bv\cdot\nabla u+\sigma_t u=g$ with $u|_{\Gamma_-}=0$, and
$(P_\Sigma\phi)(\bx,\bv)=\int_{\Sd}p(\bv,\bv')\phi(\bx,\bv')\,d\bv'$ is the angular scattering. Inserting the characteristic solution
\[
  (Tg)(\bx,\bv)=\int_0^{\tau_-(\bx,\bv)}
  \exp\!\Big(-\!\int_0^\ell\sigma_t(\bx-s\bv)\,ds\Big)\,g(\bx-\ell\bv,\bv)\,d\ell
\]
of~\eqref{EQ:Characteristic Sol} into $\cS\psi=P_\Sigma T(\sigma_s\psi)$ and applying the polar change of variables~\eqref{EQ:Polar} (set $\by=\bx-\ell\bv'$, so $\ell=|\bx-\by|$, $\bv'=\omxy$, the attenuation becomes $E(\bx,\by)$, and $d\ell\,d\bv'=d\by/(\nud|\bx-\by|^{d-1})$) recovers the kernel $\sigma_s(\by)W(\bx,\by)p(\bv,\omxy)$ of $\cS$, with $\psi$ sampled at $(\by,\omxy)$.

Let us use the notation $\|\phi\|_w^2=\int_\Omega w(\bx)\int_{\Sd}\phi^2\,d\bv\,d\bx$. The $\sigma_s$-weighted
norm $\|\cdot\|_{\sigma_s}$ is equivalent to the standard one since
$0<\sigma_{s,\min}\le\sigma_s\le\sigma_1$. We need two elementary bounds.

First, we multiply $\bv\cdot\nabla u+\sigma_t u=g$ by $u$ and integrate over $\Omega\times\Sd$. This gives us $\int\bv\cdot\nabla u\,u=\tfrac12\int_{\Gamma_+}(\bv\cdot\bn)u^2\ge0$ (as $u|_{\Gamma_-}=0$). Therefore, we have
$\|u\|_{\sigma_t}^2\le\int gu\le\|g\|_{\sigma_t^{-1}}\|u\|_{\sigma_t}$~\cite{DaLi-Book93-6,Agoshkov-Book12}, that is
\begin{equation}\label{EQ:Bound 1}
\|Tg\|_{\sigma_t}\le\|g\|_{\sigma_t^{-1}}
\end{equation}
Second, since $p\ge0$ and $\int_\Sd p(\bv,\bv')\,d\bv=\int_\Sd p(\bv,\bv')\,d\bv'=1$ (rotational invariance), Cauchy--Schwarz against
the kernel gives immediately $\int_\Sd(P_\Sigma\phi)^2\,d\bv\le\int_\Sd\phi^2\,d\bv'$ pointwise in $\bx$. Therefore, we have, for any spatial weight $w$,
\begin{equation}\label{EQ:Bound 2}
\|P_\Sigma\phi\|_w\le\|\phi\|_w
\end{equation} 
These bounds, together with the sub-criticality assumption~\eqref{EQ:Assumption}, then allow us to deduce that
\begin{multline}
  \|\cS\psi\|_{\sigma_s}
  =\|P_\Sigma\,T(\sigma_s\psi)\|_{\sigma_s}
  \underset{by~\eqref{EQ:Bound 2}}{\le}\|T(\sigma_s\psi)\|_{\sigma_s}
  \\ \underset{by~\eqref{EQ:Assumption}}{\le}\|T(\sigma_s\psi)\|_{\sigma_t}
  \underset{by~\eqref{EQ:Bound 1}}{\le}\|\sigma_s\psi\|_{\sigma_t^{-1}}
  =\Big(\!\int\tfrac{\sigma_s}{\sigma_t}\,\sigma_s\,\psi^2\Big)^{1/2}
  \underset{by~\eqref{EQ:Assumption}}{\le}\sqrt{\gamma_0}\,\|\psi\|_{\sigma_s}.
\end{multline}
Thus $\|\cS\|_{\sigma_s}\le\sqrt{\gamma_0}$, and since $P_\perp$ is orthogonal in $\|\cdot\|_{\sigma_s}$ (its weight is spatial), $\|H\|=\|P_\perp\cS P_\perp\|\le\|\cS\|\le\sqrt{\gamma_0}<1$. The Neumann series $\sum_k H^k$ converges to $(\cI-H)^{-1}$ with $\|(\cI-H)^{-1}\|\le(1-\sqrt{\gamma_0})^{-1}$, and norm equivalence
transfers~\eqref{EQ:Hcontraction} to the standard $L^2$ norm. 
\end{proof} 
The rate $\sqrt{\gamma_0}$ in the above lemma is clearly not sharp, as in the isotropic limit, the self-adjointness of Lemma~\ref{LEM:Symm} gives the better constant $\gamma_0$. However, it already yields invertibility for every anisotropy.

The invertibility of $I-H$ in the previous Lemma allows us to eliminate $\psi_\perp$ in~\eqref{EQ:psiperp-Eq} and~\eqref{EQ:U Eq} to get a single closed equation for $U$ only. This is stated in the following theorem.
\begin{theorem}\label{THM:ANISO Closed}
Let the phase function be nonnegative, $p\ge0$. Then, under the sub-criticality assumption~\eqref{EQ:Assumption}, ~\eqref{EQ:psiperp-Eq} and~\eqref{EQ:U Eq} imply that $U$ solves
\begin{equation}\label{EQ:ANISO-Closed}
  \Big(\cI - K - K(\cI-H)^{-1}G\Big)\,U
  = \Big(K(\cI-H)^{-1}H + K\Big)\big(\sigma_s^{-1}f\big).
\end{equation}
In the isotropic limit $a\equiv0$ the operators $H$ and $G$ vanish
and~\eqref{EQ:ANISO-Closed} reduces to~\eqref{EQ:RTE Integral Form}.
\end{theorem}
\begin{proof}
By Lemma~\ref{LEM:Hcontraction}, $\cI-H$ has bounded inverse. ~\eqref{EQ:psiperp-Eq} then implies that $\psi_\perp=(\cI-H)^{-1}\big(GU+H(\sigma_s^{-1}f)\big)$.
Substituting this into~\eqref{EQ:U Eq} and collecting the terms in $U$ on the left and the source terms on the right gives~\eqref{EQ:ANISO-Closed}. When $a\equiv0$,~\eqref{EQ:HG-ops} gives $H=G=0$, leading to $(\cI-K)U=K(\sigma_s^{-1}f)$, which is~\eqref{EQ:RTE Integral Form}.
\end{proof}
Equation~\eqref{EQ:ANISO-Closed} is a genuine closed second-kind equation for $U$ alone. The price of anisotropy is the correction operator $K(\cI-H)^{-1}G$, whose action requires one auxiliary solve of the angularly resolved problem~\eqref{EQ:psiperp-Eq}. The anisotropy is thus confined to the zero-mean fluctuation $\psi_\perp$, which is typically small and smooth, so the correction is a mild perturbation of the isotropic operator. This is the operator-level analogue of the asymptotic ``diffusion-correction'' closure method~\cite{CoAn-JQSRT23}, but without any small-mean-free-path assumption. 

We now characterize quantitatively the perturbation caused by anisotropy. Let
\[
    \Keff:=K+K(\cI-H)^{-1}G, \qquad \Kefff:=K+K(\cI-H)^{-1}H\,,
\]
be the two operators in~\eqref{EQ:ANISO-Closed}. We can show that anisotropy is an $\cO(\|a\|_\infty)$ perturbation, precisely stated in the following proposition (all operator norms taken in the $\sigma_s$-weighted $L^2$ norm of Lemma~\ref{LEM:Symm}).
\begin{proposition}\label{PROP:ANISO Pert}
Under assumption~\eqref{EQ:Assumption} and $p\ge0$, the correction operators obey the bounds $\|G\|_{\sigma_s},\,\|H\|_{\sigma_s}\le \gamma_0\,\|a\|_{\infty}$ . Moreover, we have
\begin{equation}\label{EQ:ANISO Pert}
  \big\|\Keff-K\big\|_{\sigma_s}
   = \big\|K(\cI-H)^{-1}G\big\|_{\sigma_s}
   \le \frac{\gamma_0^{2}\,\|a\|_{\infty}}{1-\sqrt{\gamma_0}} = \cO(\|a\|_\infty),
\end{equation}
and the same bound holds for $\Kefff-K$.
\end{proposition}
\begin{proof}
From~\eqref{EQ:HG-ops}, the angular factor obeys $|a(\bv,\omxy)|\le\|a\|_\infty$ while the spatial operator $K: g\mapsto\int_\Omega\sigma_s(\by)W(\bx,\by)g(\by)\,d\by$ has norm $\le \gamma_0$ by Lemma~\ref{LEM:Symm} (in the $\sigma_s$-weighted $L^2$ norm). Therefore, we have $\|G\|_{\sigma_s},\|H\|_{\sigma_s}\le \gamma_0\|a\|_\infty$.
By Lemma~\ref{LEM:Hcontraction}, $\|(\cI-H)^{-1}\|_{\sigma_s}\le(1-\sqrt{\gamma_0})^{-1}$. The leading factor $K$ in $K(\cI-H)^{-1}G$
acts on the angle-resolved field with the same kernel bound $\gamma_0$, so $\|K(\cI-H)^{-1}G\|_{\sigma_s}\le \gamma_0\cdot(1-\sqrt{\gamma_0})^{-1}\cdot \gamma_0\|a\|_\infty$, which is~\eqref{EQ:ANISO Pert}.
\end{proof}
 
For the Henyey--Greenstein phase function $p_{\rm HG}$~\eqref{EQ:HG} in $d=2$, $\|a\|_\infty=2g/(1-g)$, so the correction is $\cO(g)$ as $g\to0$ and grows only as $g\to1$ (strong forward peaking). 

The bound~\eqref{EQ:ANISO Pert} implies immediately that the anisotropic solution manifold lies in an $\cO(\|a\|_\infty)$ tube around the isotropic one and inherits its $n$-width up to that perturbation, as stated in the following corollary.
\begin{corollary}\label{COR:ANISO n-Width}
Let $U_g$ and $U_0$ solve the anisotropic~\eqref{EQ:ANISO-Closed} and the isotropic~\eqref{EQ:RTE Integral Form} densities for the same source. Then, we have
\begin{equation}\label{EQ:Ug-U0}
    \sgmsnorm{U_g-U_0}\le C(\gamma_0)\,\|a\|_\infty\,\sgmsnorm{\sigma_s^{-1}f}
\end{equation}
for some constant $C(\gamma_0)$. Moreover, we have
\begin{equation}\label{EQ:dMg-dM0}
    d_n(\cM_g)\le d_n(\cM_0)+C(\gamma_0)\|a\|_\infty\sup_{f}\sgmsnorm{\sigma_s^{-1}f}\,.
\end{equation}
\end{corollary}
\begin{proof}
It is straightforward to verify, using~\eqref{EQ:RTE Integral Form} and~\eqref{EQ:ANISO-Closed} that $U_g-U_0$ solves
\[
  (\cI-K)(U_g-U_0)=(\Kefff-K)(\sigma_s^{-1}f)+(\Keff-K)\,U_g .
\]
By Proposition~\ref{PROP:ANISO Pert}, $\|\Keff-K\|_{\sigma_s}$ and $\|\Kefff-K\|_{\sigma_s}$ are $\cO(\|a\|_\infty)$ with a constant depending only on $\gamma_0$ (operator norms in $\sgmsnorm{\cdot}$). Moreover, $U_g$ is bounded uniformly in $g$,
$\sgmsnorm{U_g}\le C(\gamma_0)\,\sgmsnorm{\sigma_s^{-1}f}$, since the joint scattering--transport operator $\cS$ is a contraction (Lemma~\ref{LEM:Hcontraction}), so the coupled system~\eqref{EQ:psiperp-Eq}-\eqref{EQ:U Eq} is well-posed with a $g$-independent constant. Therefore, both right-hand terms are $\cO(\|a\|_\infty)\,\sgmsnorm{\sigma_s^{-1}f}$. Meanwhile, by Lemma~\ref{LEM:Symm}, $\cI-K$ is self-adjoint and positive definite in $\sgmsip{\cdot}{\cdot}$ with $\sgmsnorm{(\cI-K)^{-1}}\le(1-\gamma_0)^{-1}$. This gives~\eqref{EQ:Ug-U0}.

For the $n$-width, fix an optimal $n$-dimensional subspace $V_n$ for $\cM_0$. Every source yields $U_g^f=U_0^f+e^f$ with $\sgmsnorm{e^f}\le C(\gamma_0)\,\|a\|_\infty\,\sgmsnorm{\sigma_s^{-1}f}$, so
$\operatorname{dist}(U_g^f,V_n)\le\operatorname{dist}(U_0^f,V_n)+\sgmsnorm{e^f}$. Taking the supremum over admissible $f$ gives~\eqref{EQ:dMg-dM0}.
\end{proof}
This result shows that the anisotropic density needs ``essentially the isotropic number of modes'' for a moderate anisotropy factor $g$.

\section{Density-based reduced-order models}
\label{SEC:ROM}

We start with the isotropic case. Our main focus is on the source-to-density map:
\begin{equation}\label{EQ:f2U}
  \Lambda: f \longmapsto U^f := (\cI-K)^{-1}K(\sigma_s^{-1}f),
\end{equation}
with fixed optical parameters $\sigma_a$ and $\sigma_s$.

\subsection{Snapshots and the POD basis} 

Let $N$ be the number of spatial degrees of freedom of the full-order discretization, so that a discrete density is a vector $\bU\in\bbR^{N}$. Given $N_s$ training sources $\{f_i\}_{i=1}^{N_s}$, we compute the corresponding full-order densities $\bU^{f_i}\in\bbR^N$ and assemble the snapshot matrix
\begin{equation}\label{EQ:Snapshot}
  \cF = \big[\,\bU^{f_1}\ \bU^{f_2}\ \cdots\ \bU^{f_{N_s}}\,\big]\in\bbR^{N\times N_s}.
\end{equation}
A truncated SVD $\cF = \cU_\cF \Sigma_\cF \cV_\cF^\top$, with singular values $s_1\ge s_2\ge\cdots$, produces the $L^2$-orthonormal POD basis
\begin{equation}\label{EQ:POD}
  \Rb = \cU_\cF(:,1{:}r)\in\bbR^{N\times r},\qquad \Rb^\top\Rb=I_r,
\end{equation}
where the rank $r$ is chosen by an energy criterion: $r$ is the smallest integer with
\begin{equation}\label{EQ:Energy}
  \frac{\sum_{i=1}^{r}s_i^2}{\sum_{i=1}^{N_s}s_i^2}\ge\eta,
\end{equation}
for a prescribed threshold $\eta$ (e.g.\ $\eta=95\%$--$99\%$). POD is optimal in the mean-square sense: among all rank-$r$ subspaces, the column space of $\Rb$
minimizes $\sum_i \|\bU^{f_i}-\Rb\Rb^\top\bU^{f_i}\|_2^2 = \sum_{i>r} s_i^2$.
The decay rate of $s_i$ is therefore a direct, computable diagnostic of reducibility, and is fast for the density $U$ even when it would be slow for the full field $u$. We document this decay across regimes in~\Cref{SEC:Num}.

\subsection{Projection strategies}
\label{SUB:proj}

Let $\bK\in\bbR^{N\times N}$ be the full-order discretization of $K$, so the full-order problem for a (possibly unseen) source $\tilde f$ is
\begin{equation}\label{EQ:FOM Lin}
  (\bI-\bK)\,\bU = \boldsymbol\phi, \qquad \boldsymbol\phi = \bK(\sigma_s^{-1}\tilde f).
\end{equation}
Seeking $\bU\approx \Rb\,c$ with reduced coordinates $c\in\bbR^r$, we implement two projections, mirroring the standard taxonomy and the source/solution asymmetry of the integral operator.
\begin{description}
\item[\textbf{M1: Solution-space projection.}] Substitute
  $\bU\approx\Rb c$ into~\eqref{EQ:FOM Lin} and solve the overdetermined
  $N\times r$ system in the least-squares sense, to have
  \begin{equation}\label{EQ:M1}
    c = \argmin_{c\in\bbR^r}\ \|(\bI-\bK)\Rb\,c - \boldsymbol\phi\|_2
      = \big[(\bI-\bK)\Rb\big]^{\dagger}\boldsymbol\phi .
  \end{equation}
  This minimizes the full residual. However, it requires an application of $\Rb^\top(\bI-\bK)^\top$ to $\phi$ at the online stage. %
  
\item[\textbf{M2: Galerkin projection.}] Test against the trial space itself,
  \begin{equation}\label{EQ:M2}
    \Rb^\top(\bI-\bK)\Rb\,c = \Rb^\top\boldsymbol\phi
    \quad\Longleftrightarrow\quad (\,I_r - \widehat\bK\,)\,c = \Rb^\top\boldsymbol\phi,
    \qquad \widehat\bK := \Rb^\top\bK\Rb\in\bbR^{r\times r}.
  \end{equation}
  The reduced operator $\widehat\bK$ is precomputable. The online solve is $\cO(r^3)$.   
\end{description}
In both cases, once $c$ is found, the reduced density is recovered by $\Urom=\Rb c$. The aim is $r\ll N$, so small that a fast solver for~\eqref{EQ:FOM Lin} is unnecessary online. The dominant offline cost is the $N_s$ full-order solves and one SVD.

\subsection{Stability and error analysis}
\label{SUBSEC:Error}

Because the governing equation is a second-kind Fredholm equation with a
contractive kernel (see Proposition~\ref{PROP:Contraction}), the reduced problem inherits well-posedness and a clean error bound. We work in $\ell^2(\bbR^N)$ and write $\kappa$ for the discrete contraction constant. When $\sigma_s$ is constant, $\bK$ is symmetric and $\kappa=\|\bK\|_2=\rho(\bK)\le \gamma_0$ by the discrete analogue of~\eqref{EQ:Contraction}. For variable $\sigma_s$, the same bound holds in the $\sigma_s$-weighted inner product of Lemma~\ref{LEM:Symm}, in which $\bK$ is self-adjoint, so that $\kappa=\rho(\bK)\le \gamma_0<1$. The error bounds below are stated in this contraction norm, which is equivalent to the standard $\ell^2$ norm.

\begin{proposition}\label{PROP:ROM-Wellposed}
For the Galerkin ROM~\eqref{EQ:M2}, $\widehat\bK=\Rb^\top\bK\Rb$ satisfies
$\|\widehat\bK\|_2\le\|\bK\|_2=\kappa<1$. Hence $I_r-\widehat\bK$ is invertible with $\|(I_r-\widehat\bK)^{-1}\|_2\le
(1-\kappa)^{-1}$, and the reduced solution is unique and uniformly bounded.
\end{proposition}

The ROM error is quasi-optimal: it is, up to the regime-dependent constant
$\tfrac{1+\kappa}{1-\kappa}$, the best approximation error of $\bU$ in the
reduced space, precisely stated in the following theorem.
\begin{theorem}\label{THM:Cea}
Let $\bU=(\bI-\bK)^{-1}\boldsymbol\phi$ be the full-order density and
$\Urom=\Rb c$ the Galerkin ROM~\eqref{EQ:M2}. Let
$\Pi=\Rb\Rb^\top$ be the orthogonal projector onto the reduced space. Then
\begin{equation}\label{EQ:Cea}
  \|\bU-\Urom\|_2 \ \le\ \frac{1+\kappa}{1-\kappa}\,
  \big\|\bU-\Pi\bU\big\|_2\,.
\end{equation}
\end{theorem}
\begin{proof}
Let $\bA=\bI-\bK$. The Galerkin solution satisfies the projected residual
condition $\Rb^\top(\boldsymbol\phi-\bA\Urom)=0$, while $\bA\bU=\boldsymbol\phi$. Hence, $\Rb^\top\bA(\bU-\Urom)=0$. Now, we decompose $\bU-\Urom=(\bI-\Pi)\bU +
(\Pi\bU-\Urom)$ with $e:=\Pi\bU-\Urom\in\operatorname{range}\Rb$. Applying
$\Rb^\top\bA$ and using $\Rb^\top\bA(\bU-\Urom)=0$ gives
$\Rb^\top\bA\,e = -\Rb^\top\bA(\bI-\Pi)\bU = \Rb^\top\bK(\bI-\Pi)\bU$ (since
$\Rb^\top\bA(\bI-\Pi)= \Rb^\top(\bI-\Pi) - \Rb^\top\bK(\bI-\Pi)
= -\Rb^\top\bK(\bI-\Pi)$, as $\Rb^\top(\bI-\Pi)=0$). Because $e\in\operatorname{range}\Rb$,
$\Rb^\top\bA e=(I_r-\widehat\bK)\,\Rb^\top e$. Therefore, 
$\|e\|_2=\|\Rb^\top e\|_2\le (1-\kappa)^{-1}\kappa\,\|(\bI-\Pi)\bU\|_2$. Combining with $\|(\bI-\Pi)\bU\|_2$ and the triangle inequality, we get $\|\bU-\Urom\|_2\le \big(1+\tfrac{\kappa}{1-\kappa}\big)\|(\bI-\Pi)\bU\|_2
=\tfrac{1}{1-\kappa}\|(\bI-\Pi)\bU\|_2$, which is sharper than and implies
\eqref{EQ:Cea}.
\end{proof}
The bound is the contraction-specialized form of the standard reduced-basis (C\'ea-type) quasi-optimality estimate~\cite{HeRoSt-Book16,QuMaNe-Book16}, here with the closed-form constant $\tfrac{1+\kappa}{1-\kappa}$ supplied by the contraction.

\begin{corollary}\label{COR:SVD Control}
If the test density $\bU$ lies (to tolerance) in the span of the training snapshots, then $\|(\bI-\Pi)\bU\|_2$ is controlled by the discarded singular values: for a training snapshot $\bU^{f_i}$,
$\|(\bI-\Pi)\bU^{f_i}\|_2\le s_{r+1}$, and averaged over the training set, $\frac1{N_s}\sum_i\|(\bI-\Pi)\bU^{f_i}\|_2^2=\frac1{N_s}\sum_{j>r}s_j^2$. Thus the relative ROM error scales like $s_{r+1}/s_1$ up to the amplification factor $(1-\kappa)^{-1}$.
\end{corollary}
Corollary~\ref{COR:SVD Control} makes precise the practical rule used in~\Cref{SEC:Num}: choose $r$ from the energy criterion~\eqref{EQ:Energy}, and expect the relative error to track the singular-value tail, with a mild $(1-\kappa)^{-1}$ inflation that grows as scattering strengthens ($\kappa\to1$). For sources far outside the training distribution, the bound is governed by the Kolmogorov $r$-width of the solution manifold rather than by $s_{r+1}$. %

The constant $\tfrac{1+\kappa}{1-\kappa}$ in~\eqref{EQ:Cea} is generic. It can be halved (in the exponent) by exploiting the self-adjointness of Lemma~\ref{LEM:Symm}, computing the POD in the $\sigma_s$-weighted inner product $\sgmsip{\cdot}{\cdot}$ of Lemma~\ref{LEM:Symm}, so that $\Rb$ is $\sigma_s$-orthonormal and the discrete operator $\widehat\bK$ inherits symmetry. $\bA=\bI-\bK$ is then self-adjoint and positive-definite in $\sgmsip{\cdot}{\cdot}$, with energy norm $\|v\|_\bA^2:=\sgmsip{\bA v}{v}$.

\begin{corollary}\label{COR:ROM Energy}
With the $\sigma_s$-weighted POD and Galerkin projection~\eqref{EQ:M2}, the reduced
density is the $\bA$-orthogonal projection of $\bU$, $\|\bU-\Urom\|_\bA=\min_{v\in\operatorname{ran}\Rb}\|\bU-v\|_\bA$, and
\begin{equation}\label{EQ:Cea-Energy}
  \sgmsnorm{\bU-\Urom}\ \le\ \sqrt{\tfrac{1+\kappa}{1-\kappa}}\;
  \sgmsnorm{\bU-\Pi_{\sigma_s}\bU},
\end{equation}
where $\Pi_{\sigma_s}$ is the $\sigma_s$-orthogonal projector onto $\operatorname{ran}\Rb$. The amplification is thus the \emph{square root} of that
in~\eqref{EQ:Cea}.
\end{corollary}
\begin{proof}
Galerkin orthogonality $\sgmsip{\bA(\bU-\Urom)}{v}=0$ for all $v\in\operatorname{ran}\Rb$ is the definition of the $\bA$-orthogonal projection, giving energy-optimality. We then convert the norms to get $\sgmsnorm{\bU-\Urom}^2\le\lambda_{\min}(\bA)^{-1}\|\bU-\Urom\|_\bA^2
\le\lambda_{\min}(\bA)^{-1}\|\bU-\Pi_{\sigma_s}\bU\|_\bA^2
\le\frac{\lambda_{\max}(\bA)}{\lambda_{\min}(\bA)}\sgmsnorm{\bU-\Pi_{\sigma_s}\bU}^2$,
and $\operatorname{spec}(\bA)\subset[1-\kappa,1+\kappa]$ by~\eqref{EQ:Real Spec}.
\end{proof}

Finally, we present a certified, online a~posteriori error estimator as a distinctive benefit of the contraction structure. The contraction structure makes the stability constant known in closed form (i.e. $\kappa\le \gamma_0=\sup_\bx\sigma_s/\sigma_t$ is read
off the coefficients), so a rigorous error bound is available online without any
eigenvalue or inf--sup estimation. For any reduced density $\Urom=\Rb c$ (from M1,
or M2), we define the residual and the estimator
\begin{equation}\label{EQ:Estimator}
  \br(\Urom):=\boldsymbol\phi-(\bI-\bK)\Urom,
  \qquad
  \Delta(\Urom):=\frac{\|\br(\Urom)\|_2}{1-\kappa}.
\end{equation}
The following is the standard reduced-basis residual error
estimator~\cite{HeRoSt-Book16,QuMaNe-Book16}, specialized to the second-kind
contraction structure, where the stability constant $(1-\kappa)^{-1}$ is known in
closed form rather than estimated.
\begin{theorem}\label{THM:aposteriori}
The estimator~\eqref{EQ:Estimator} brackets the true error,
\begin{equation}
  \frac{\|\br(\Urom)\|_2}{1+\kappa}\ \le\ \|\bU-\Urom\|_2\ \le\ \Delta(\Urom),
  \qquad\text{hence}\qquad
  1\ \le\ \frac{\Delta(\Urom)}{\|\bU-\Urom\|_2}\ \le\ \frac{1+\kappa}{1-\kappa}.
\end{equation}
\end{theorem}
\begin{proof}
Since $(\bI-\bK)\bU=\boldsymbol\phi$, the error satisfies
$(\bI-\bK)(\bU-\Urom)=\br$. The upper bound uses
$\|(\bI-\bK)^{-1}\|_2\le(1-\kappa)^{-1}$. The lower bound uses
$\|\br\|_2=\|(\bI-\bK)(\bU-\Urom)\|_2\le(1+\kappa)\|\bU-\Urom\|_2$. The effectivity
follows.
\end{proof}

\subsection{The anisotropic reduced model}
\label{SUBSEC:ANISO ROM}

We give two matching reductions in the anisotropic case, both of which deliver $U$.

\subsubsection{Intrusive reduction}
 
In this first route, we discretize~\eqref{EQ:psiperp-Eq}-\eqref{EQ:U Eq} as
\begin{equation*}
  (\bI-\bH)\boldsymbol\Phi = \bG\,\bU + \bH\,\bQ,
  \qquad
  (\bI-\bK)\,\bU = \bD\,\boldsymbol\Phi + \bD\,\bQ,
\end{equation*}
where $\bU\in\bbR^N$ is the density, $\boldsymbol\Phi\in\bbR^{NM_a}$ collects the fluctuation $\psi_\perp$ at the $N$ nodes and $M_a$ angles, $\bQ=\sigma_s^{-1}\bff$ is the rescaled source, and $\bK$, $\bD$, $\bG$, $\bH$ denote the full discrete blocks assembled in Section~\ref{SUBSEC:Nystrom ANISO}. We build two POD bases from
snapshots: $\Rb_U\in\bbR^{N\times r}$ for the density and
$\Rb_\Phi\in\bbR^{NM_a\times r_\Phi}$ for the fluctuation. Let us write the reduced blocks $\wh\bK=\Rb_U^\top\bK\Rb_U$, $\wh \bD=\Rb_U^\top\bD\Rb_\Phi$, $\wh \bG=\Rb_\Phi^\top\bG\Rb_U$, and $\wh \bH=\Rb_\Phi^\top\bH\Rb_\Phi$, then the block Petrov-Galerkin (Galerkin) ROM reads as
\begin{equation}\label{EQ:ANISO-ROM-A}
\begin{aligned}
  (\,I_{r_\Phi}-\wh\bH)\,\bc_\Phi
    &= \wh\bG \,\bc_U + \Rb_\Phi^\top \bH\,\bQ,\\
  (\,I_r-\wh\bK)\,\bc_U
    &= \wh\bD\,\bc_\Phi + \Rb_U^\top\bD\,\bQ,
\end{aligned}
\qquad \Urom=\Rb_U\,\bc_U\,.
\end{equation}
This is a small $(r+r_\Phi)\times(r+r_\Phi)$ system. We eliminate $\bc_\Phi$ to reproduce the reduced analogue of the Schur complement~\eqref{EQ:ANISO-Closed}. The reduction is effective because the fluctuation $\psi_\perp$ is itself smooth and small (it is the $a$-weighted average~\eqref{EQ:psiperp-Def} of $u$), so $r_\Phi$ is modest, often smaller than $r$ in our numerical tests. Since the joint scattering--transport operator $\cS$ of the coupled system is a contraction, $\|\cS\|\le\sqrt{\gamma_0}<1$ for every anisotropy (\Cref{LEM:Hcontraction}, its discretization inheriting this up to the quadrature consistency error), the block Petrov--Galerkin projection of~\eqref{EQ:ANISO-ROM-A} is well-posed. Proposition~\ref{PROP:ROM-Wellposed} and Theorem~\ref{THM:Cea} therefore extend to the reduced Schur operator. Therefore, this route is \emph{certified}.

\subsubsection{Non-intrusive data-driven reduction}

When only density snapshots $\{\bU^{f_i}\}$ are available, for example, from a black-box discrete-ordinates or Monte~Carlo solver, we model $U$ directly. The exact equation~\eqref{EQ:ANISO-Closed} has the second-kind form $(\cI-K_{\rm eff})U=K_{\rm eff}'(\sigma_s^{-1}f)$ with the (unknown to the
black box) effective operators $K_{\rm eff}=K+K(\cI-H)^{-1}G$ and $K_{\rm eff}'=K(\cI-H)^{-1}H+K=K(\cI-H)^{-1}$. Positing a reduced surrogate $\Urom=\Rb_U\,\bc$, we infer reduced operators $\widehat\bM,\widehat\bN$ by regularized least-squares operator inference~\cite{PeWi-CMAME16}:
\begin{equation}\label{EQ:Operator Inf}
  (\widehat\bM,\widehat\bN)
  = \argmin_{\bM,\bN}\ \sum_{i=1}^{N_s}
    \big\| \Rb_U^\top\bU^{f_i} - \bM\,\Rb_U^\top\bU^{f_i}
           - \bN\,\Rb_U^\top\boldsymbol\phi_i \big\|_2^2 + \lambda\,(\|\bM\|_F^2+\|\bN\|_F^2),
\end{equation}
with $\boldsymbol\phi_i=\bK(\sigma_s^{-1}f_i)$, so that online $\bc=(I_r-\widehat\bM)^{-1}\widehat\bN\,\Rb_U^\top\tilde{\boldsymbol\phi}$. The Tikhonov parameter $\lambda$ controls the
conditioning of the inference. 

This recovers the isotropic Galerkin ROM exactly when $g=0$ (as then $\widehat\bM=\widehat\bK$, $\widehat\bN=I_r$) and otherwise learns the anisotropy correction $K(\cI-H)^{-1}G$ from data, at the cost of forfeiting the certified bound of \Cref{THM:Cea}. The prediction error remains controlled by the operator-learning error. To characterize this prediction error, let $\widehat\bM_\star=\Rb_U^\top\bK_{\rm eff}\Rb_U$ be the Galerkin projection of the effective operator of \Cref{THM:ANISO Closed}, and let $\widehat\bN_\star$ be the reduced load operator characterized by $\widehat\bN_\star\,\Rb_U^\top\boldsymbol\phi_i=\Rb_U^\top\Kefff(\sigma_s^{-1}f_i)$ on the training sources (so that $\widehat\bN_\star=I_r$ in the isotropic limit $\Kefff=K$, matching $\widehat\bN=I_r$ above), and let $\bc_\star=(I-\widehat\bM_\star)^{-1}\widehat\bN_\star\Rb_U^\top\tilde{\boldsymbol\phi}$ be the coordinates they produce for an unseen source. We have the following result.

\begin{proposition}\label{PROP:Operator Inf}
Suppose the inferred operators satisfy $\|\widehat\bM-\widehat\bM_\star\|_2\le\delta_M$,
$\|\widehat\bN-\widehat\bN_\star\|_2\le\delta_N$, and that the Tikhonov penalty keeps $\|\widehat\bM\|_2\le\kappa'<1$. Then the predicted density
$\Urom=\Rb_U(I-\widehat\bM)^{-1}\widehat\bN\,\Rb_U^\top\tilde{\boldsymbol\phi}$
obeys
\begin{equation}\label{EQ:Operator Inf-bound}
  \|\bU-\Urom\|_2 \ \le\ \|\bU-\Pi_U\bU\|_2
  \ +\ \frac{\delta_M\,\|\bc_\star\|_2+\delta_N\,\|\Rb_U^\top\tilde{\boldsymbol\phi}\|_2}{1-\kappa'},
\end{equation}
with $\Pi_U=\Rb_U\Rb_U^\top$.
\end{proposition}
\begin{proof}
Let us write $\be=\bc-\bc_\star$. The subtraction of the two reduced systems gives 
\[
(I-\widehat\bM)\be=(\widehat\bM-\widehat\bM_\star)\bc_\star
+(\widehat\bN-\widehat\bN_\star)\Rb_U^\top\tilde{\boldsymbol\phi}\,.
\]
This gives $\|\be\|_2\le(1-\kappa')^{-1}(\delta_M\|\bc_\star\|_2+\delta_N\|\Rb_U^\top\tilde{\boldsymbol\phi}\|_2)$.
Since $\Rb_U\bc_\star=\Pi_U\bU$ up to the projected-operator consistency error, the triangle inequality with $\|\bU-\Rb_U\bc_\star\|_2=\|\bU-\Pi_U\bU\|_2$ yields~\eqref{EQ:Operator Inf-bound}.
\end{proof}
The term $\|\bU-\Pi_U\bU\|_2$ is the projection (best-approx) error. Hence, the data-driven model is as accurate as the projection allows, plus a term linear in the operator-learning errors $\delta_M,\delta_N$ (which vanish in the isotropic limit, where the inference is exact).

\subsection{Offline-online algorithms}
\label{SUB:Alg}

We now summarize the standard offline and online stages of the RTE ROMs we proposed. The isotropic and the anisotropic reductions differ only in which reduced variable is compressed, how many bases are built, and which reduced operators are formed. We summarize the common skeleton in~\Cref{ALG:Offline,ALG:Oline} and use~\Cref{Tab:ROM Cases} to instantiate it for each model.

\begin{algorithm}[!htbp]
\caption{Offline stage: model reduction}
\label{ALG:Offline}
\begin{algorithmic}[1]
\renewcommand{\algorithmicrequire}{\textbf{Input:}}
\renewcommand{\algorithmicensure}{\textbf{Output:}}
\Require Training sources $\{f_i\}_{i=1}^{N_s}$, energy threshold $\eta$.
\State Select model and reduced variables (\Cref{Tab:ROM Cases}).
\For{$i=1,\dots,N_s$}
  \State Compute the full-order snapshots of the reduced variables for $f_i$.
  \State Append them to the snapshot matrix.
\EndFor
\State Take SVD of snapshot matrix and set POD basis $\Rb_\bullet$ by energy criterion~\eqref{EQ:Energy}.
\State Form the corresponding reduced operators (\Cref{Tab:ROM Cases}).  
\Ensure Bases and reduced operators as listed in \Cref{Tab:ROM Cases}.
\end{algorithmic}
\end{algorithm}

\begin{algorithm}[!htbp]
\caption{Online stage: predict $\Urom^{\tilde f}$ for an unseen source $\tilde f$}
\label{ALG:Oline}
\begin{algorithmic}[1]
\renewcommand{\algorithmicrequire}{\textbf{Input:}}
\renewcommand{\algorithmicensure}{\textbf{Output:}}
\Require Reduced data from \Cref{ALG:Offline} and source $\tilde f$.
\State Assemble the reduced source by projecting $\sigma_s^{-1}\tilde f$. 
\State Solve the small reduced system for the coordinates $\bc$: a projection M1--M2
       \eqref{EQ:M1}--\eqref{EQ:M2} (isotropic), the block system~\eqref{EQ:ANISO-ROM-A}
       (A) or the inferred system~\eqref{EQ:Operator Inf} (B).
\State \textbf{return} $\Urom^{\tilde f}=\Rb_U\,\bc_U$ (the density / zeroth-moment block).
\Ensure Reduced density $\Urom^{\tilde f}\approx\bU^{\tilde f}$.
\end{algorithmic}
\end{algorithm}

\begin{table}[!htbp]
\centering
\begin{tabular}{lcccc}
\toprule
Model & variable(s) & basis/bases & reduced operator(s) & online size \\
\midrule
Isotropic (M1) & $U$ & $\Rb$ & $[(\bI-\bK)\Rb]^\dagger$ & $r$ \\
Isotropic (M2) & $U$ & $\Rb$ & $\widehat\bK=\Rb^\top\bK\Rb$ & $r$ \\
Aniso A (coupled) & $(U,\psi_\perp)$ & $\Rb_U,\Rb_\Phi$ & $\widehat\bH,\widehat\bG,\widehat\bD,\widehat\bK$ & $r+r_\Phi$ \\
Aniso B (data-driven) & $U$ & $\Rb_U$ & $\widehat\bM,\widehat\bN$ (inferred) & $r$ \\
\bottomrule
\end{tabular}
\caption{Instantiation of the generic offline/online algorithms for each reduced model. The ``online size" column reports the dimension of the reduced solve.}
\label{Tab:ROM Cases}
\end{table}

\section{Full-order discretization}
\label{SEC:FOM}

We now provide some implementation details on the discretization of the full-order models. We follow closely~\cite{ReZhZh-JCP19,FaAnYi-JCP19}.

\subsection{Nystr\"om discretization of the isotropic integral operator}
\label{SUB:nystrom}

On $\Omega=[0,L_x]\times[0,L_y]$ we use a uniform Cartesian grid
$\bx_{ij}=(i\Delta x, j\Delta y)$, $i=0,\dots,N_x-1$, $j=0,\dots,N_y-1$, with
$\Delta x=L_x/(N_x-1)$, $\Delta y=L_y/(N_y-1)$, and $N=N_xN_y$. To each node, we attach the node-centered control volume
\begin{equation*} 
  C_{ij} = [x_i-\ell_i,\,x_i+r_i]\times[y_j-b_j,\,y_j+t_j]\subset\Omega,
\end{equation*}
where $\ell_i,r_i\in\{0,\Delta x/2\}$ and $b_j,t_j\in\{0,\Delta y/2\}$ are halved
at the boundary, so the cell weight $\omega_{ij}=|C_{ij}|$ equals
$\Delta x\Delta y$ at interior nodes, $\tfrac12\Delta x\Delta y$ at edges, and
$\tfrac14\Delta x\Delta y$ at corners (a built-in midpoint/trapezoidal weighting). Flattening $(i,j)\mapsto n=\iota(i,j)=iN_y+j+1$, we write $\bx_n,C_n,\omega_n$ and $U_n\approx U(\bx_n)$, $f_n\approx f(\bx_n)$. The Nystr\"om approximation treats $U$ and $f$ as constant on each cell.

\paragraph{Homogeneous, constant-coefficient case.} Here
\eqref{EQ:RTE Integral Form} reads, with $r=|\bx-\by|$,
\begin{equation*}
  \Big(\cI-\tfrac{\sigma_s}{2\pi}K_0\Big)U=\tfrac{1}{2\pi}K_0 f
  \ \ \text{with}\ \
  (K_0 g)(\bx)=\int_\Omega G(\bx,\by)\,g(\by)\,d\by,\quad
  G(\bx,\by)=\frac{e^{-\sigma_t r}}{r}.
\end{equation*}
We assemble the matrix $\bK_0\in\bbR^{N\times N}$,
$(\bK_0)_{mn}\approx\int_{C_n}G(\bx_m,\by)\,d\by$, by a three-way split that
respects the weak ($1/r$) singularity and the rapid near-field variation of $G$:
\begin{enumerate}
\item \emph{Far field} ($n$ outside a $(2\rho{+}1)\times(2\rho{+}1)$ patch around
  $m$): one-point weighted node evaluation,
  $(\bK_0)_{mn}=\omega_n\,G(\bx_m,\bx_n)=\omega_n\,e^{-\sigma_t|\bx_m-\bx_n|}/|\bx_m-\bx_n|$.
\item \emph{Diagonal} ($m=n$): polar integration about $\bx_n$, $\by=\bx_n+\rho(\cos\theta,\sin\theta)$,
  with $R_n(\theta)=\sup\{\rho\ge0:\bx_n+\rho(\cos\theta,\sin\theta)\in C_n\}$,
  \begin{equation}\label{EQ:Diag}
    (\bK_0)_{nn}
    =\int_0^{2\pi}\!\!\int_0^{R_n(\theta)}\! e^{-\sigma_t\rho}\,d\rho\,d\theta
    =\int_0^{2\pi}\frac{1-e^{-\sigma_t R_n(\theta)}}{\sigma_t}\,d\theta,
  \end{equation}
  which removes the singularity exactly (the Jacobian $\rho$ cancels
  $1/r$) and is evaluated by the periodic trapezoidal rule in $\theta$.
\item \emph{Near field} ($m\neq n$ inside the patch): tensor-product
  Gauss--Legendre quadrature on the clipped cell $C_n=[a_n,b_n]\times[c_n,d_n]$,
  \begin{equation}\label{EQ:Near Field}
    (\bK_0)_{mn}\approx J_n\sum_{\alpha=1}^{n_g}\sum_{\beta=1}^{n_g}
    w_\alpha w_\beta\,\frac{e^{-\sigma_t|\bx_m-\by^{(n)}_{\alpha\beta}|}}{|\bx_m-\by^{(n)}_{\alpha\beta}|},
    \quad J_n=\tfrac14(b_n-a_n)(d_n-c_n),
  \end{equation}
  with $(\xi_\alpha,w_\alpha)$ the Gauss--Legendre nodes/weights on $[-1,1]$ and
  $\by^{(n)}_{\alpha\beta}$ the mapped quadrature points. This resolves the steep
  but integrable near-field profile that the one-point far-field rule misses.
\end{enumerate}
The resulting linear system is
$\big(\cI-\tfrac{\sigma_s}{2\pi}\bK_0\big)\bU=\tfrac{1}{2\pi}\bK_0\bff$.

\paragraph{Inhomogeneous case.} For variable coefficients only $G$ changes: the
constant rate $\sigma_t r$ is replaced by the ray integral
$\tau(\bx,\by)=\int_0^{|\bx-\by|}\sigma_t(\bx-s\,\omxy)\,ds$, which we evaluate by
$n_a$-point Gauss--Legendre quadrature along the segment,
\begin{equation}\label{EQ:tau-h}
  \tau_h(\bx,\by)=|\bx-\by|\sum_{k=1}^{n_a}\varpi_k\,\sigma_{t,h}\!\big(\bx+\eta_k(\by-\bx)\big),
  \qquad G_h(\bx,\by)=\frac{e^{-\tau_h(\bx,\by)}}{r},
\end{equation}
$\sigma_{t,h}$ being the bilinear interpolant of the nodal $\sigma_t$ and
$(\eta_k,\varpi_k)$ the Gauss--Legendre rule on $[0,1]$. The same three-way split
applies with $\sigma_s(\by)$ carried inside the cell integral, giving the
assembled $\bK$ used to form $\widehat\bK=\Rb^\top\bK\Rb$.

The three-way split is designed to preserve, at the discrete level, the two analytic properties on which the ROM rests: the contraction $\|K\|<1$ and the compactness of Lemma~\ref{LEM:Smoothing}. Both hold because the diagonal and near-field rules use only nonnegative weights and integrate the kernel consistently across the $1/r^{d-1}$ singularity.

$\bK$ is dense ($\cO(N^2)$ storage), so direct assembly is feasible only for moderate $N$. The full-order integral solve is not our target. The ROM sidesteps this online: after the one-time offline assembly and projection, only $\widehat\bK\in\bbR^{r\times r}$ is used. For large $N$ one would assemble and apply $\bK$ matrix-free via fast multipole method (FMM)~\cite{ReZhZh-JCP19,FaAnYi-JCP19}. The projection $\Rb^\top\bK\Rb$ then costs $r$ fast applies.

\subsection{Discretization of the anisotropic integral system}
\label{SUBSEC:Nystrom ANISO}

We now discretize the coupled fluctuation system~\eqref{EQ:psiperp-Eq}-\eqref{EQ:U Eq}, reusing the spatial control volumes $\{\bx_n,C_n,\omega_n\}_{n=1}^N$ and the three-way quadrature split above. The new ingredient is the angular dependence carried by $\psi_\perp$ and by the anisotropic kernel $a(\bv,\omxy)$ (specialized below to the rotationally-invariant Henyey-Greenstein model used in the experiments).

\paragraph{Angular grid and the off-grid difficulty.} Use the uniform midpoint grid on $\bbS^1$,
$\theta_\ell=(\ell-\tfrac12)\Delta\theta$, $\Delta\theta=2\pi/M_a$, $\bv_\ell=(\cos\theta_\ell,\sin\theta_\ell)$, $\ell=1,\dots,M_a$, with weights $\mu_\ell=1/M_a$. Let $\Phi_{n\ell}=\psi_\perp(\bx_n,\bv_\ell)$ and $f_{n\ell}=f(\bx_n,\bv_\ell)$. The operators $H,G$ in~\eqref{EQ:HG-ops} evaluate the kernel at the geometric direction $\varphi_m(\by)=\arg(\bx_m-\by)$, which is generally not an angular grid node. We therefore use periodic piecewise-linear interpolation: for $\varphi$ with $\theta_{\ell_0}\le\varphi<\theta_{\ell_0+1}$ (periodic indexing), $\alpha(\varphi)=(\theta_{\ell_0+1}-\varphi)/\Delta\theta$, $\beta(\varphi)=1-\alpha(\varphi)$, define the interpolation row $L(\varphi)\in\bbR^{1\times M_a}$ with the only nonzero entries $L_{\ell_0}=\alpha(\varphi)$, $L_{\ell_0+1}=\beta(\varphi)$, so that $L(\varphi)\,\bz=\alpha\,z_{\ell_0}+\beta\,z_{\ell_0+1}$ interpolates angular data $\bz\in\bbR^{M_a}$ at $\varphi$.

\paragraph{Discrete anisotropic kernel.} At angle $\varphi$ define the normalized discrete Henyey-Greenstein weights and their zero-mean part,
\begin{equation}\label{EQ:Disc HG}
  p^{(M_a)}_\ell(\varphi)=\frac{\frac{1-g^2}{1+g^2-2g\cos(\theta_\ell-\varphi)}}
       {\frac{1}{M_a}\sum_{k=1}^{M_a}\frac{1-g^2}{1+g^2-2g\cos(\theta_k-\varphi)}},
  \quad
  a^{(M_a)}_\ell(\varphi)=p^{(M_a)}_\ell(\varphi)-1,
  \quad \tfrac{1}{M_a}\sum_{\ell=1}^{M_a} a^{(M_a)}_\ell(\varphi)=0,
\end{equation}
where the denominator enforces the discrete normalization $\frac{1}{M_a}\sum_\ell
p^{(M_a)}_\ell=1$, so the zero-angular-mean property~\eqref{EQ:HG-zeromean} holds exactly at the discrete level. With the variable-coefficient attenuation $W_h(\bx,\by)=e^{-\tau_h(\bx,\by)}/(2\pi|\bx-\by|)$ from~\eqref{EQ:tau-h}, the exact per-cell integral blocks are
\begin{equation*}
\begin{aligned}
  (K^\ast)_{mn} &= \int_{C_n}\sigma_s(\by)\,W_h(\bx_m,\by)\,d\by, \quad
  (D^\ast)_{mn} = \int_{C_n}\sigma_s(\by)\,W_h(\bx_m,\by)\,L(\varphi_m(\by))\,d\by,\\
  (G^\ast)_{mn} &= \int_{C_n}\sigma_s(\by)\,W_h(\bx_m,\by)\,a^{(M_a)}(\varphi_m(\by))\,d\by, \\
  (H^\ast)_{mn} &= \int_{C_n}\sigma_s(\by)\,W_h(\bx_m,\by)\,a^{(M_a)}(\varphi_m(\by))\otimes L(\varphi_m(\by))\,d\by,
\end{aligned}
\end{equation*}
where $a^{(M_a)}(\varphi)\in\bbR^{M_a}$ is the column of weights~\eqref{EQ:Disc HG}, $L(\varphi)\in\bbR^{1\times M_a}$ the interpolation row, and $\otimes$ the outer product. The assembled matrices $\bK,\bD,\bG,\bH$ are the quadrature realizations of the exact per-cell blocks $K^\ast,D^\ast,G^\ast,H^\ast$ above. Thus $\bK\in\bbR^{N\times N}$ acts on the density $\bU$.
$\bD\in\bbR^{N\times NM_a}$ maps the angular fluctuation to its
$\omxy$-sampled contribution to $U$ (the discrete realization of the isotropic operator $K$ acting on the angularly resolved field), and
$\bG\in\bbR^{NM_a\times N}$, $\bH\in\bbR^{NM_a\times NM_a}$ realize $G$ and $H$. Each block is assembled by the same far-field/diagonal-polar/near-field-Gauss--Legendre split as in the isotropic case.

\section{Numerical experiments}
\label{SEC:Num}

We now present some numerical simulations to demonstrate the performance of the ROMs we developed. We focus on problems in the two-dimensional unit square $\Omega=(0,1)^2$. We cover the domain with a uniform $N_x=N_y=65$ grid (that is, $\Delta x=\Delta y=1/64$), i.e.\ $N=N_xN_y=4225$ spatial degrees of freedom in the sense of \Cref{SEC:FOM}. We consider training data from two families of source functions. The first is \emph{localized} isotropic Gaussian sources of the form
\begin{equation}\label{EQ:Source Gaussian}
  f_i(\bx;\bx_i,\varsigma^2) = \exp\!\Big(-\frac{|\bx-\bx_i|^2}{2\varsigma^2}\Big),
\end{equation}
with centers $\bx_i$ drawn uniformly over $\Omega$ (approaching point sources as $\varsigma\to0$). The second family is \emph{global} source functions built from random Fourier features,
\begin{equation}\label{EQ:Source Fourier}
f_i(\bx)=\sum_{\bk\in[-64,64]^2\cap\bbZ^2} \beta_{i,\bk} e^{i\pi \bk \cdot\bx}, \quad \beta_{i,-\bk}=\overline{\beta_{i,\bk}}
\end{equation}
with random coefficients $\{\beta_{i,\bk}\}$ drawn from Gaussian with $\bbE|\beta_{i,\bk}|\sim |\bk|^{-2}$ for $\bk\neq 0$ (the $\bk=0$ mean is set to a fixed positive constant) for each $i$, enriching the data with smooth, domain-wide patterns. We scale the random Fourier sources~\eqref{EQ:Source Fourier} so that $f_i\ge 0$ to make them physically relevant (even though this is mathematically unnecessary). Note that the Gaussian sources~\eqref{EQ:Source Gaussian} are bounded and holomorphic in their low-dimensional parameter (the center $\bx_i$), so Proposition~\ref{PROP:n-Width} directly guarantees a fast $n$-width decay. The random Fourier sources~\eqref{EQ:Source Fourier} are likewise holomorphic in their coefficients, but their parameter dimension is large. For that family, the compressibility of the density manifold is governed less by Proposition~\ref{PROP:n-Width} than by the $|\bk|^{-2}$ spectral decay of the coefficients together with the smoothing of Lemma~\ref{LEM:Smoothing} (as the experiments below confirm).

Throughout this section, we parametrize the scattering regime by the contraction constant $\gamma_0=\sup_\bx\sigma_s/\sigma_t$ of~\eqref{EQ:Assumption}, which is the quantity that governs both the well-posedness and the error amplification $(1-\gamma_0)^{-1}$
of \Cref{THM:Cea,COR:SVD Control}. Unless stated otherwise, we fix $\sigma_a=0.1$ and choose $\sigma_s$ to realize $\gamma_0\in\{0.5,0.8,0.95,0.99\}$ (transport-, intermediate-, scattering-dominated, and near-critical regimes, with amplifications
$2,5,20,100$).
 
The reduced rank $r$ in all the simulations is chosen by the energy criterion~\eqref{EQ:Energy} with $\eta$ as stated. We report the relative error of the ROMs by averaging over two disjoint source sets: the $N_s$ \emph{training} sources $\{f_i\}_{i=1}^{N_s}$
used to build the reduced basis, and a held-out set of $N_{\rm test}$ \emph{unseen} test sources $\{\tilde f_i\}_{i=1}^{N_{\rm test}}$ drawn from the same family but not used in the offline stage. To be precise, we write $U_F^{f}$ and $U_R^{f}$ for the full-order and reduced densities driven by source $f$. This yields the \emph{in-sample} and \emph{out-of-sample} errors
\begin{equation}\label{EQ:eps-in-out}
  \eps_{\rm in} := \frac{1}{N_s}\sum_{i=1}^{N_s}
    \frac{\|U_F^{f_i}-U_R^{f_i}\|_{L^2}}{\|U_F^{f_i}\|_{L^2}},
  \qquad
  \eps_{\rm out} := \frac{1}{N_{\rm test}}\sum_{i=1}^{N_{\rm test}}
    \frac{\|U_F^{\tilde f_i}-U_R^{\tilde f_i}\|_{L^2}}{\|U_F^{\tilde f_i}\|_{L^2}}.
\end{equation}
The in-sample error $\eps_{\rm in}$ measures how well the reduced basis represents the training manifold, while the out-of-sample error $\eps_{\rm out}$ measures generalization to unseen sources.  

\subsection{ROM for isotropic medium}
We start with numerical experiments on ROM for isotropic media.
\subsubsection{Homogeneous isotropic medium}\label{sec:num-hom0-iso}

In the first set of simulations, we consider the ROM in a homogeneous isotropic medium with a fixed constant absorption coefficient $\sigma_a=0.1$ and a constant scattering coefficient $\sigma_s$, and control the scattering strength through  $\gamma_0$. We consider four different scattering strengths corresponding to $\gamma_0\in\{0.5,0.8,0.95,0.99\}$.

\begin{figure}[!htbp]
  \centering
  \includegraphics[width=0.8\linewidth]{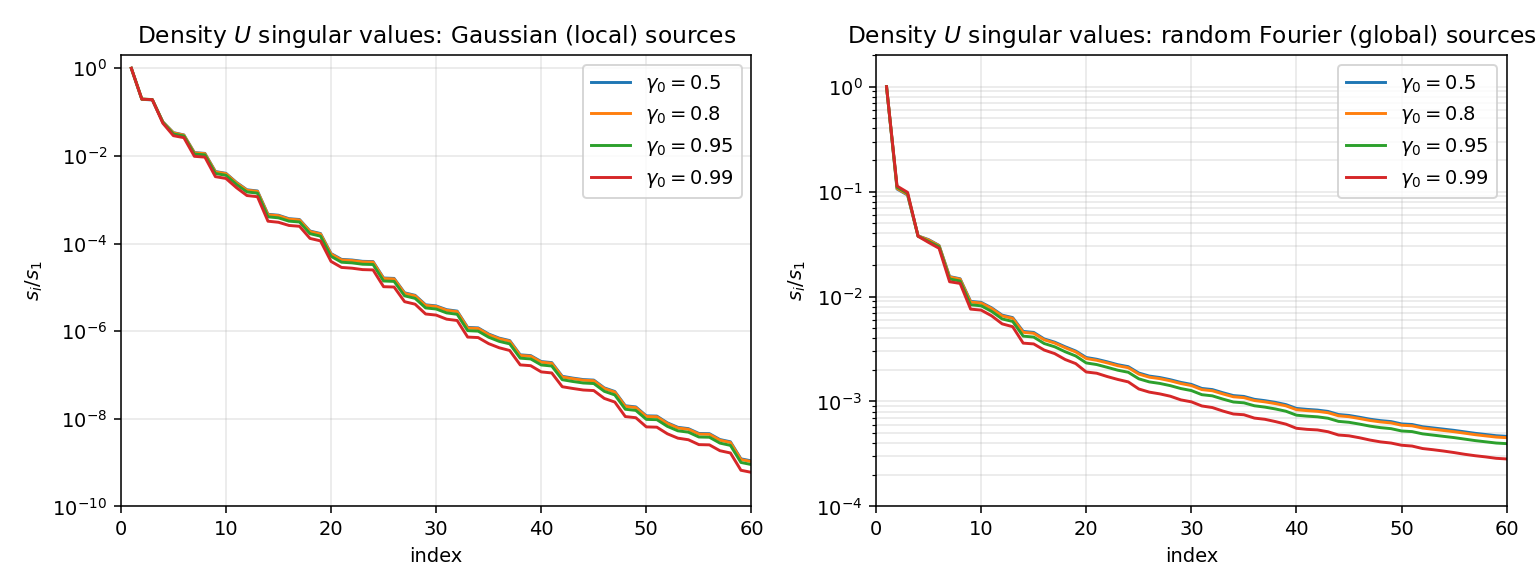}
  \caption{Normalized singular values of density snapshot matrix generated from Gaussian (left) and random Fourier (right) sources across the regime ladder $\gamma_0\in\{0.5,0.8,0.95,0.99\}$.}
  \label{FIG:Reducibility}
\end{figure}

To verify the reducibility of the density manifold, we show in~\Cref{FIG:Reducibility} the rapid, regime-robust normalized singular-value decay of the \emph{density} snapshots across the $\gamma_0$ ladder.
\begin{figure}[!htbp]
  \centering
  	\begin{tabular}{@{}c@{\hspace{0.8em}} @{\hspace{0.8em}}c@{}}
  	\textbf{\small Local Gaussian sources} & \textbf{\small Global random-Fourier sources} \\[0.2ex]
  \includegraphics[width=0.47\linewidth]{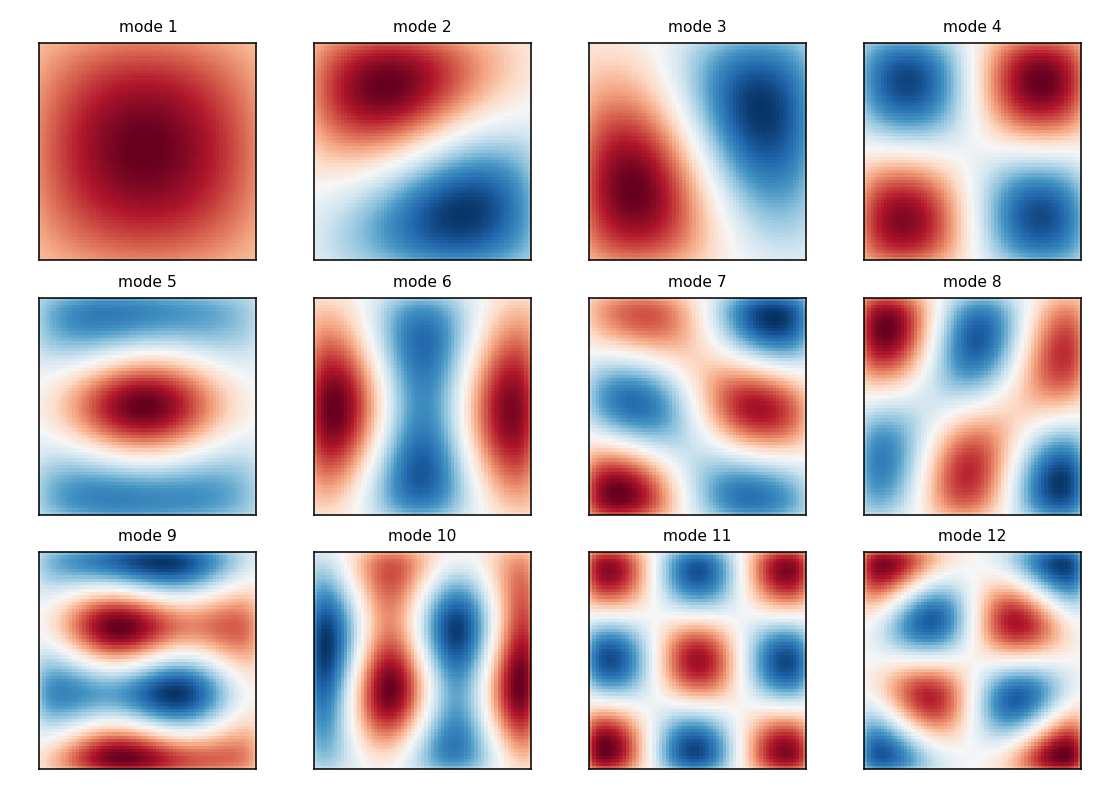}&
\includegraphics[width=0.47\linewidth]{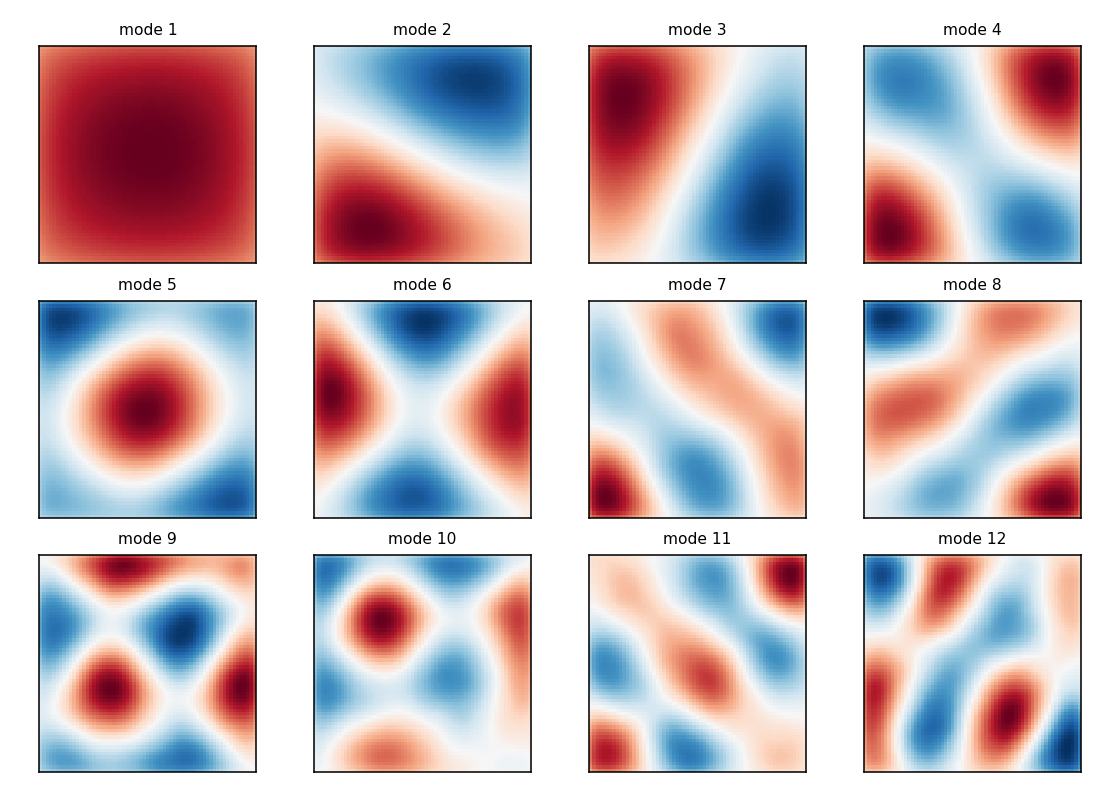}
  \end{tabular}
  \caption{Leading POD modes (left singular vectors of the data matrix in intermediate regime $\gamma_0=0.8$) for local Gaussian (left) and global random-Fourier (right) training sources.}
  \label{FIG:POD-Basis}
\end{figure}
Leading POD modes (left singular vectors of the data matrix) are shown in Figure~\ref{FIG:POD-Basis} for an intermediate regime $\gamma_0=0.8$ for local Gaussian (left) and global random-Fourier (right) training sources. The first mode is a smooth, sign-definite bump, and higher modes add increasingly oscillatory but still smooth structure.

\begin{table}[!htbp]
  \centering
  \begin{minipage}[b]{0.45\textwidth}
    \centering
\begin{tabular}{cccc}
\toprule
source $\varsigma^2$ & $r$ & $\eps_{\rm in}$ & $\eps_{\rm out}$ \\
\midrule
0.025 & 15 & 2.95e-02 & 2.68e-02 \\
0.05  &  8 & 3.13e-02 & 2.99e-02 \\
0.10  &  6 & 1.81e-02 & 1.69e-02 \\
0.20  &  3 & 2.83e-02 & 2.67e-02 \\
\bottomrule
\end{tabular}
    \caption{Dependence of errors on variance of source in intermediate regime of $\sigma_a=0.1$ and $\gamma_0=0.8$.}
\label{TAB:Variance}
  \end{minipage}
  \hfill
  \begin{minipage}[b]{0.45\textwidth}
    \centering
\begin{tabular}{ccccc}
\toprule
$\gamma_0$ & $(1-\gamma_0)^{-1}$ & $r$ & $\eps_{\rm in}$ & $\eps_{\rm out}$ \\
\midrule
0.50 & 2& 8 & 3.19e-02 & 3.05e-02\\
0.80 & 5 & 8 & 3.13e-02 & 2.99e-02\\
0.95 & 20 & 8 & 2.89e-02 & 2.75e-02\\
0.99& 100 & 8 & 2.52e-02 & 2.38e-02\\
\bottomrule
\end{tabular}
\caption{Dependence of errors on the strength of scattering. $\sigma_a=0.1$ and $\sigma_s$ is chosen to realize $\gamma_0$. 
}
\label{TAB:Sigmas}
  \end{minipage}
\end{table}

\Cref{TAB:Variance} fixes $N_s=800$ and $\eta=0.99$ and varies the source variance $\varsigma^2$: narrower sources ($\varsigma^2=0.025$) excite finer spatial features and need a larger rank ($r=15$) for the same energy, while broader sources are captured by very few modes ($r=3$ at $\varsigma^2=0.2$). This is a direct illustration of the Kolmogorov-width dependence on the source dictionary. \Cref{TAB:Sigmas} sweeps the regime ladder $\gamma_0\in\{0.5,0.8,0.95,0.99\}$ at fixed local sources (with $\varsigma^2=0.05$). We again fix $N_s=800$ and $\eta=0.99$. The rank stays at $r=8$, and the energy-truncated error stays at the few-percent level across the whole ladder, even as the worst-case amplification $(1-\gamma_0)^{-1}$ of Corollary~\ref{COR:SVD Control} grows from $2$ to $100$, demonstrating that the reduction is effective well outside the diffusive regime. The error at fixed $\eta$ is the energy-truncated value. Pushing the rank to $r=20$-$40$ drives it to $10^{-3}$--$10^{-5}$, as the effectivity study of~\Cref{fig:effectivity} in next subsection shows.

\begin{figure}[!htb]
  \centering
  \includegraphics[width=0.8\linewidth]{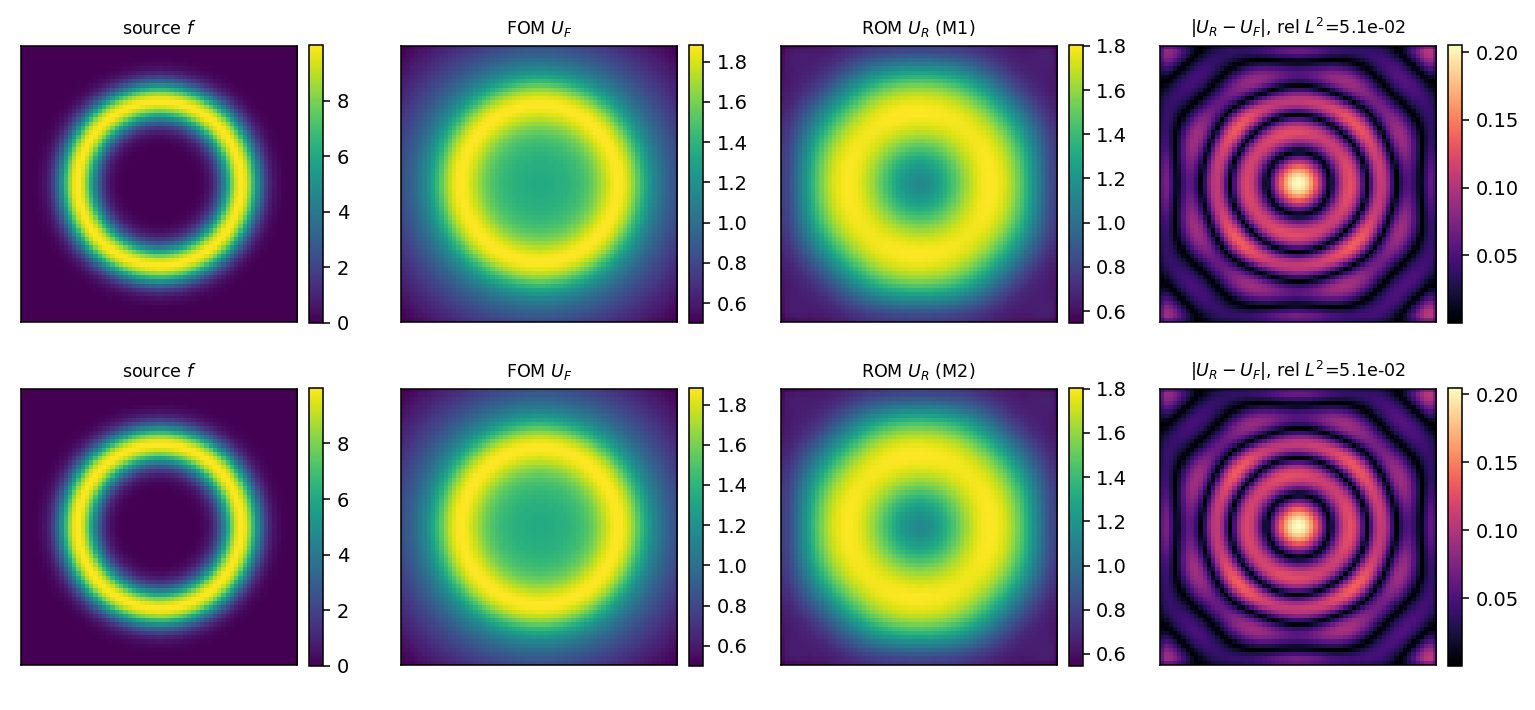}
  \caption{Testing of (M1) (top) and (M2) (bottom) on unseen ring source from local training. The mean is homogeneous with $\sigma_a=0.1$ and $\gamma_0=0.8$. The rank is selected as $r=40$. Shown from left to right are source $f$, FOM density $U_F$, ROM density $U_R$, and pointwise error $|U_R-U_F|$.
  }
  \label{FIG:Homo Reduction}
\end{figure}
In~\Cref{FIG:Homo Reduction}, we show a representative FOM/ROM comparison. Shown are online results on a source that was not included in the training data set. At $r=40$, FOM density $U_F$ and reduced density $U_R$ differ by approximately $5.1\times 10^{-2}$ in relative $L^2$ on this unseen ring source, which lies outside the local Gaussian training family, with M1 (top) and M2 (bottom) giving essentially identical errors, consistent with the energy-optimality of M2 given in Corollary~\ref{COR:ROM Energy}.

\begin{figure}[!htb]
  \centering
  \includegraphics[width=0.8\linewidth]{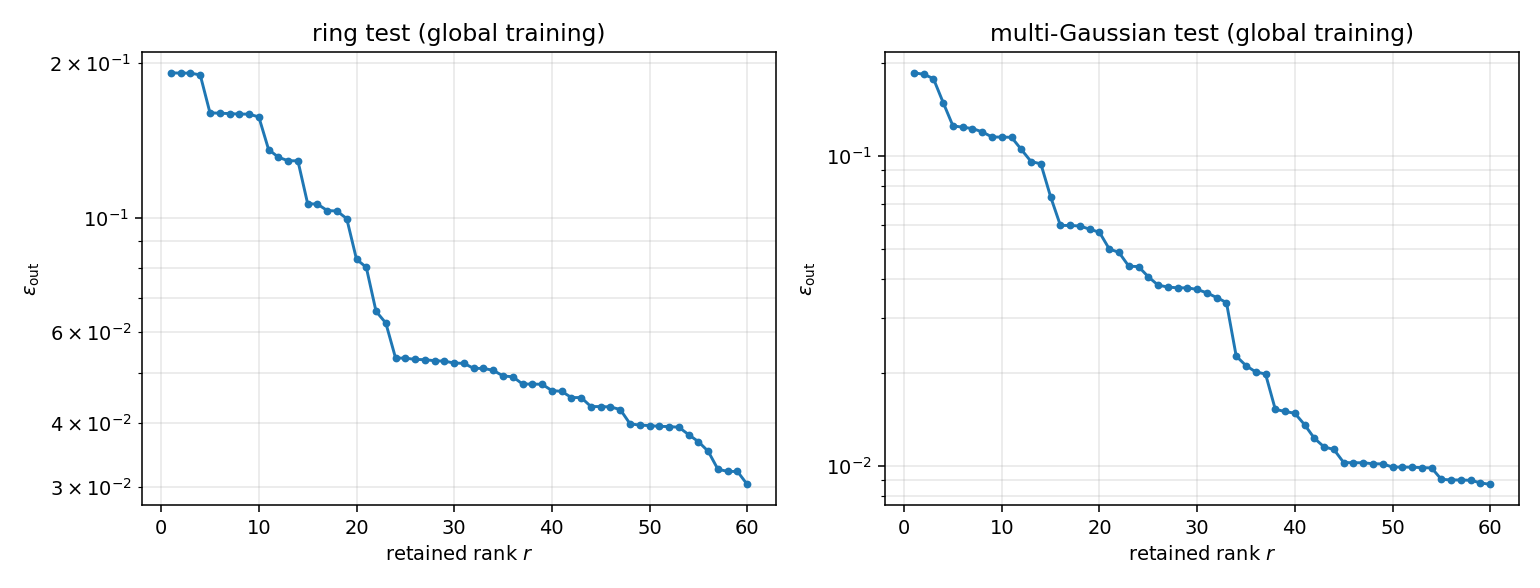}
  \caption{Out-of-sample error vs.\ retained rank for \emph{global} (random-Fourier)
  training sources on a ring test (left) and a multi-Gaussian test (right),
  $\gamma_0=0.8$.}
  \label{FIG:Error-vs-Rank}
\end{figure}
The dependence of the error on the retained rank, together with the failure of the energy criterion for global sources noted above, is shown in~\Cref{FIG:Error-vs-Rank}. In contrast to the spectral decay obtained with local training, the error falls only slowly and in steps, indicating that global-trained bases generalize poorly to localized test sources. The energy criterion (which selects $r\approx3$ at $\eta=0.99$ here), therefore, stops far too early.

M1 (solution-space least squares) and M2 (Galerkin) give essentially identical errors across all regimes, as predicted by the energy-optimality of the (weighted-)symmetric stationary operator (see Lemma~\ref{LEM:Symm} and Corollary~\ref{COR:ROM Energy}). We therefore use M2 as the default. 

\subsubsection{Inhomogeneous isotropic medium}
\label{SUBSEC:Num Inhomog}

For spatially varying coefficients, we take the smooth medium
\begin{equation}\label{EQ:Var Coeff}
  \sigma_s(\bx)=3+0.2\sin(2\pi x)\cos(2\pi y),\qquad
  \sigma_a(\bx)=0.2+0.1(x+y),\qquad \sigma_t=\sigma_s+\sigma_a,
\end{equation}
(so $\gamma_0=\sup\sigma_s/\sigma_t\approx0.94$), and a discontinuous variant in which a rectangular subregion is turned into a zero-scattering void ($\sigma_s\equiv0$, i.e.\ $\gamma_0=0$ locally). The attenuation $E(\bx,\by)$ is precomputed by ray-marching $\sigma_t$ via~\eqref{EQ:tau-h}. Unseen multi-bump sources are predicted at the few-percent level. The coefficient, FOM solution, and ROM solution are shown in~\Cref{FIG:Var Coeff} and the
error-vs-rank curves are shown in~\Cref{FIG:Var Coeff Error}. The results are entirely consistent with the constant-coefficient study: the Galerkin error decays spectrally to approximately $5\times
10^{-7}$ at $r=40$ in \emph{both} the smooth and the void media, so the framework is robust even where the kernel locally degenerates.
\begin{figure}[!htb]
	\centering
	\begin{tabular}{@{}c@{\hspace{0.8em}} @{\hspace{0.8em}}c@{}}
		\textbf{\small (a) Smooth medium} & \textbf{\small (b) Void medium} \\[0.2ex]
		\includegraphics[width=0.49\linewidth]{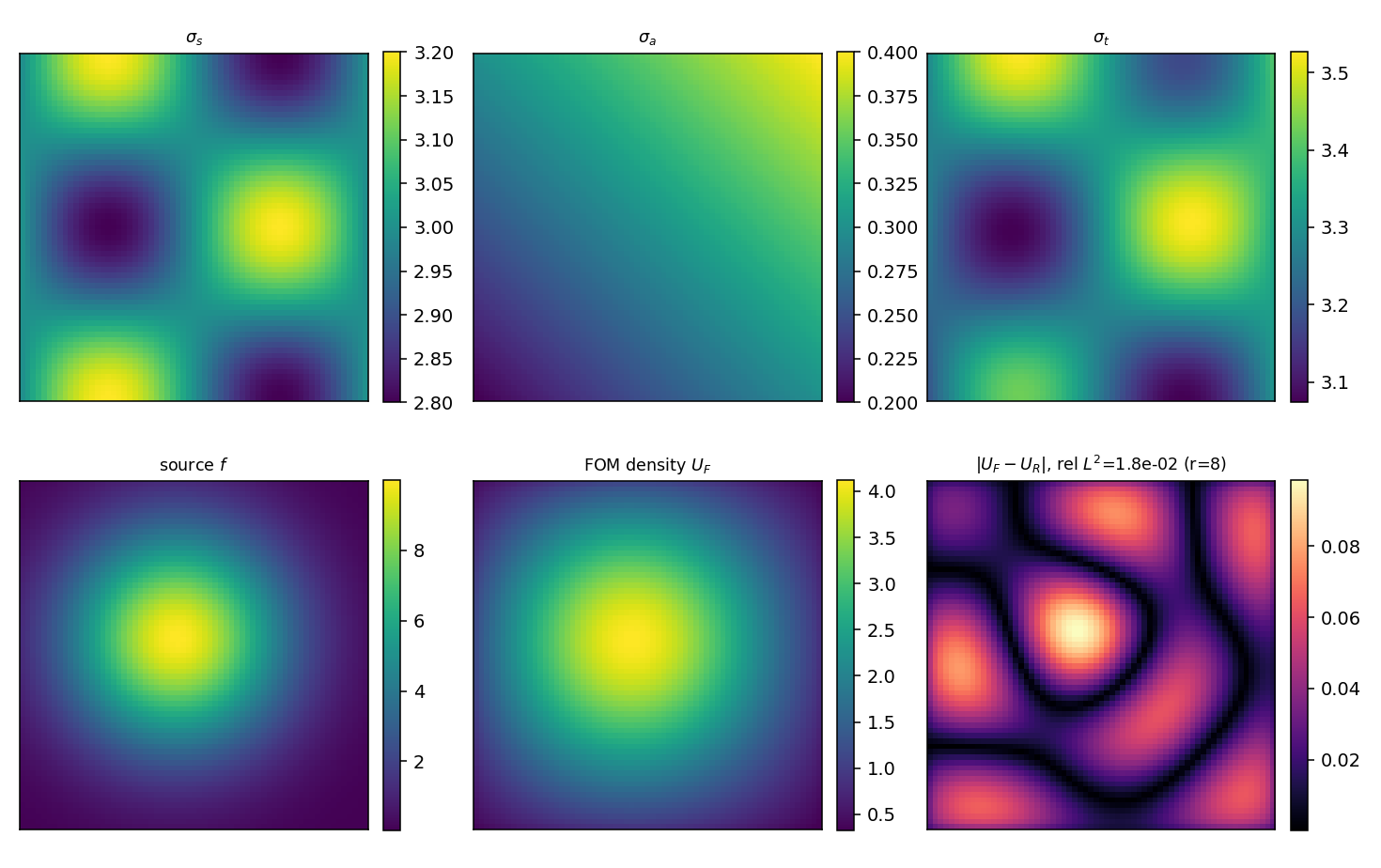} &
		\includegraphics[width=0.49\linewidth]{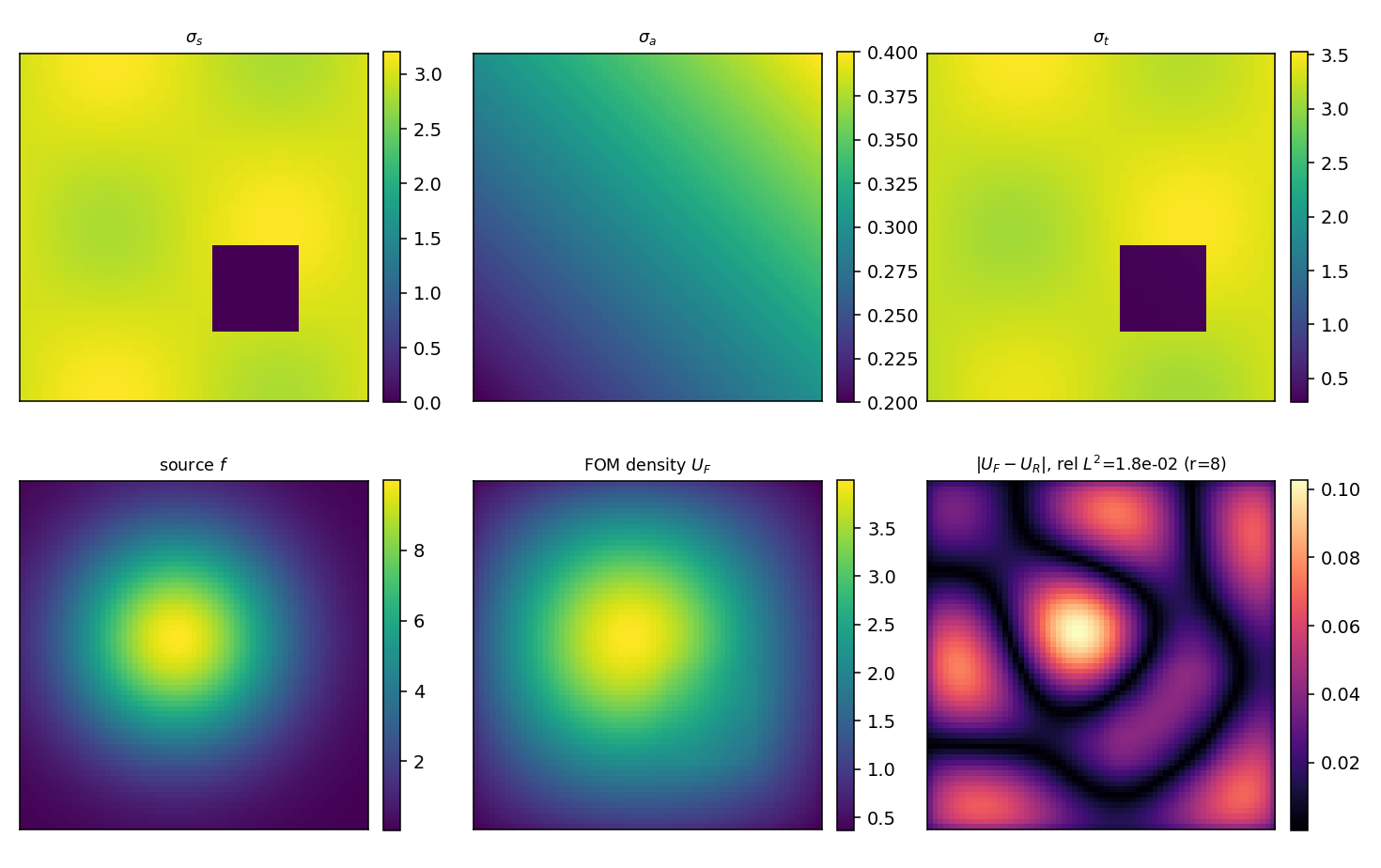}
	\end{tabular}
	\caption{Comparison of FOM and ROM (Galerkin, $r=8$) for the smooth heterogeneous medium~\eqref{EQ:Var Coeff} or for the medium with a zero-scattering void ($\sigma_s\equiv0$ in a box, $\gamma_0=0$ locally). For each subfigure, shown are $(\sigma_s, \sigma_a, \sigma_t)$ (top) and $(f, U_F, |U_F-U_R|)$ (bottom).}
	\label{FIG:Var Coeff}
\end{figure} 

\begin{figure}[!htb]
  \centering
  \includegraphics[width=0.8\linewidth]{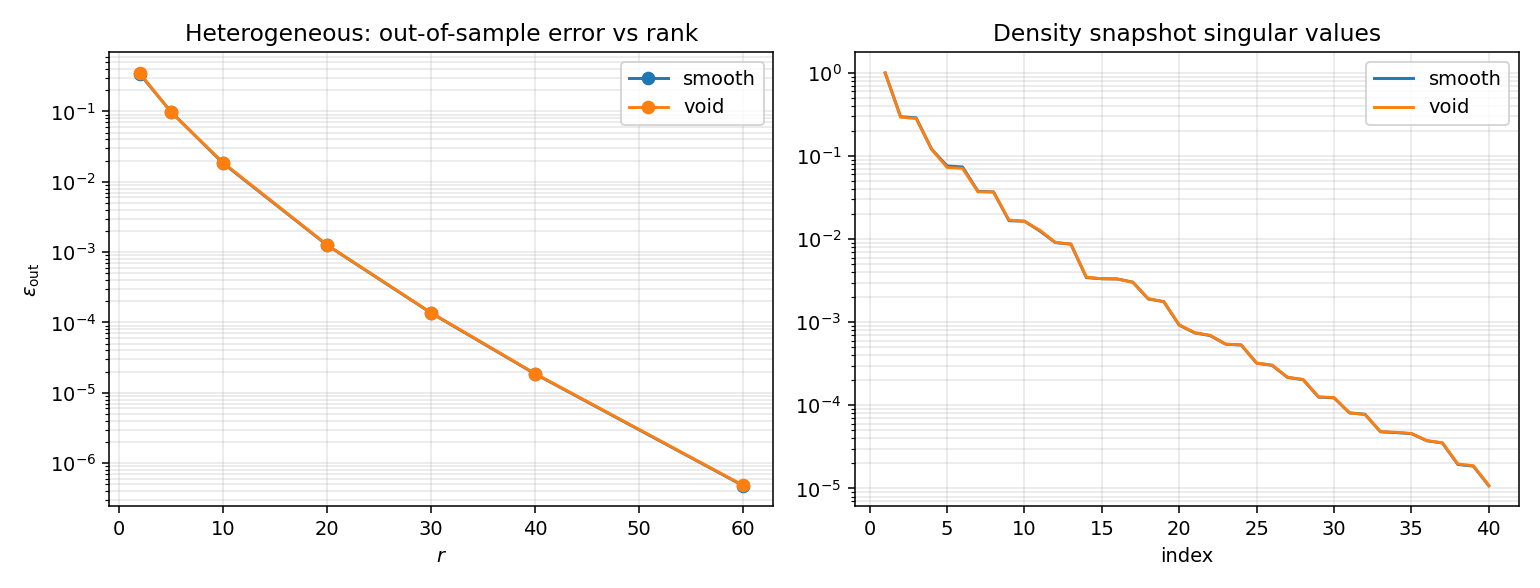}
  \caption{Out-of-sample Galerkin error vs.\ rank (left) and the corresponding density snapshot singular
  values (right) for the smooth and zero-scattering-void media of~\Cref{FIG:Var Coeff}.}
  \label{FIG:Var Coeff Error}
\end{figure}

\subsubsection{Stability and error analysis} 
\begin{figure}[!htb]
	\centering
	\includegraphics[width=0.8\linewidth]{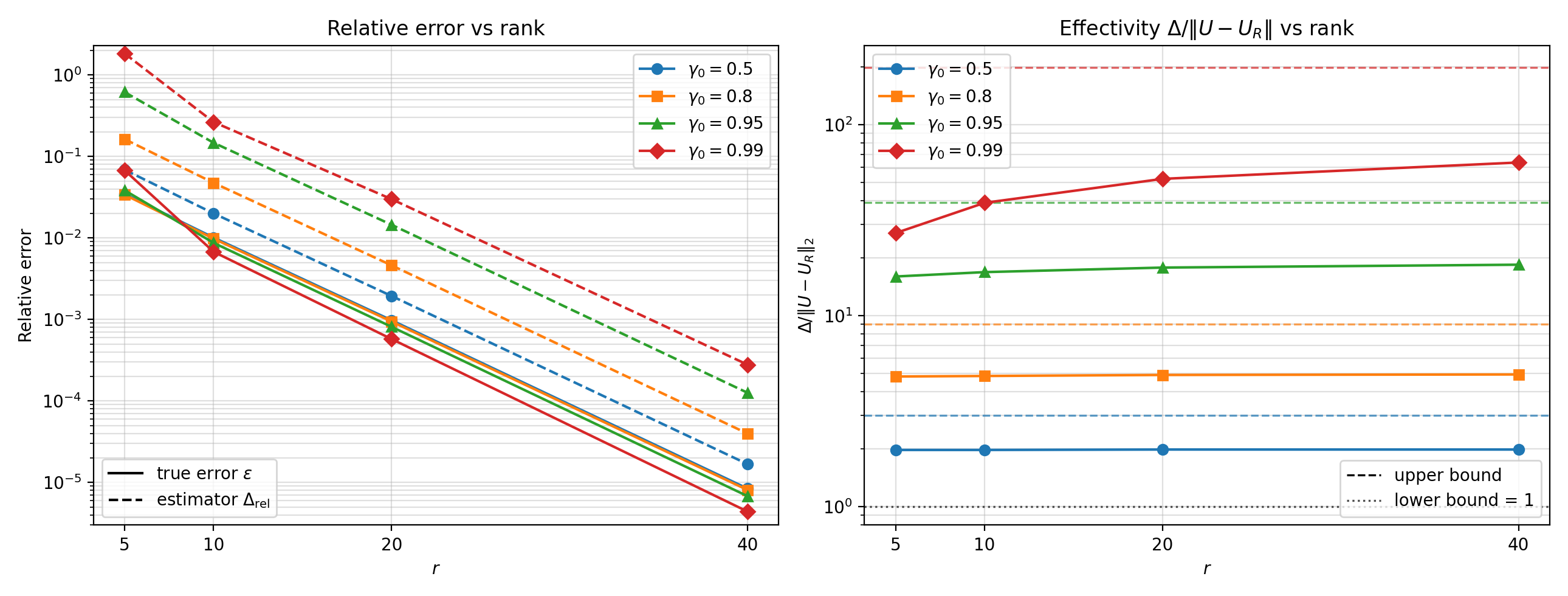}
	\caption{Left: true relative error $\eps$ (solid) and certified estimator $\Delta_{\rm rel}$ (dashed) vs.\ rank $r$ for each regime ladder $\gamma_0\in\{0.5,0.8,0.95,0.99\}$ (spectral
		decay). Right: the certified effectivity index $\Delta/\|U-\Urom\|_2$ vs.\ $r$,
		remaining within $[1,(1+\gamma_0)/(1-\gamma_0)]$ and increasing with $\gamma_0$
		(\Cref{THM:aposteriori}). Horizontal dashed lines are the theoretical lower bound $1$ and upper bound $(1+\gamma_0)/(1-\gamma_0)$ (same color as the solid lines with the same $\gamma_0$).
	}
	\label{fig:effectivity}
\end{figure}

In~\Cref{fig:effectivity}, we verify the certified residual estimator $\Delta(\Urom)=\|\br\|_2/(1-\kappa)$ of Theorem~\ref{THM:aposteriori}. Because $\kappa\le \gamma_0$ is known in closed form, $\Delta$
requires no eigenvalue estimation and is computed online in $\cO(r^2)$. 
We use the same setup as in the homogeneous isotropic medium in~\Cref{sec:num-hom0-iso}.
In the left subfigure of \Cref{fig:effectivity}, we plot, for each regime of the $\gamma_0$-ladder and several
ranks, the true relative error $\eps=\|U-\Urom\|_2/\|U\|_2$ and the relative estimator
$\Delta_{\rm rel}=\Delta/\|U\|_2$. The true relative errors $\eps$ are below the corresponding relative estimator
$\Delta_{\rm rel}$, demonstrating the first estimation in Theorem~\ref{THM:aposteriori}. In addition, the true error $\eps$ falls spectrally with $r$ (e.g.\ $3.5\!\times\!10^{-2}$ decreases to $8\!\times\!10^{-6}$ for
$r$ increases from $5$ to $40$ in the transport regime $\gamma_0=0.5$).
In the right subfigure of \Cref{fig:effectivity}, we plot the effectivity index $\Delta/\|U-\Urom\|_2$ (solid lines) and its theoretical lower bound ($=1$, dotted line) and upper bounds ($(1+\gamma_0)/(1-\gamma_0)$, dashed lines) in the second part of Theorem~\ref{THM:aposteriori}, for different regimes and ranks.
In every one of the $16$ cases the measured effectivity $\Delta/\|U-\Urom\|_2$ lies in the guaranteed interval $[1,(1+\gamma_0)/(1-\gamma_0)]$. It grows toward the ceiling as scattering strengthens, exactly tracking the $\tfrac{1+\gamma_0}{1-\gamma_0}$ conditioning. 

\begin{figure}[!htb]
	\centering
	\includegraphics[width=0.45\linewidth]{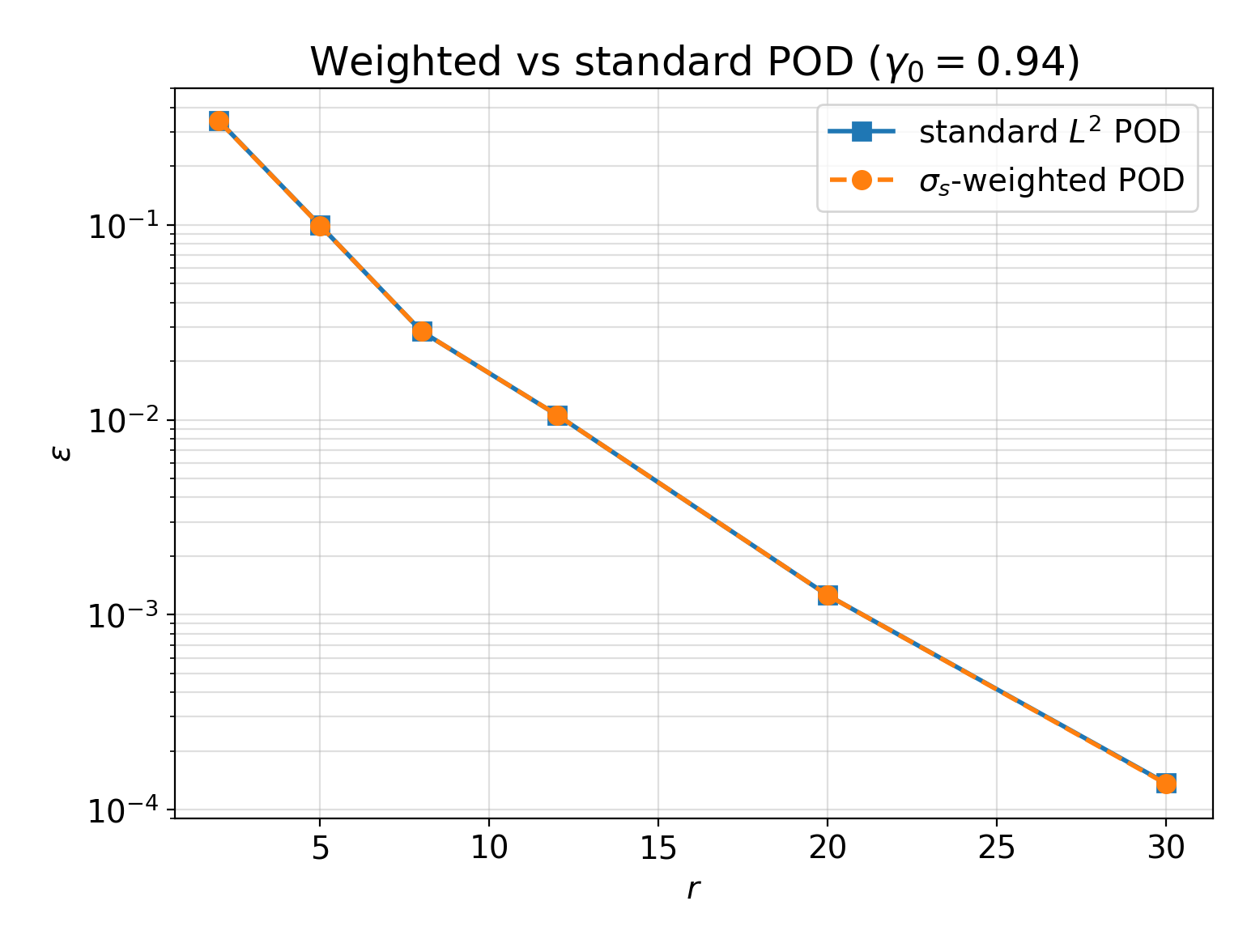}
	\caption{Standard $L^2$ vs.\ $\sigma_s$-weighted POD Galerkin error on the smooth heterogeneous medium. The two curves coincide because $\sigma_s$ is nearly constant here (see Corollary~\ref{COR:ROM Energy}).
	}
	\label{Fig:Weighted}
\end{figure}
Corollary~\ref{COR:ROM Energy} predicts that computing the POD in the $\sigma_s$-weighted inner product of Lemma~\ref{LEM:Symm} symmetrizes the reduced operator and renders the Galerkin ROM the energy-orthogonal projection, improving the worst-case error constant from
$\tfrac{1+\kappa}{1-\kappa}$ to its square root. For a homogeneous medium, the weighted
and standard inner products coincide. We therefore run this comparison on the smooth heterogeneous medium of \Cref{SUBSEC:Num Inhomog}, where $\sigma_s$ varies. \Cref{Fig:Weighted} reports the two POD inner products at matched rank. The improvement is a worst-case-constant effect, and here it is marginal: on this medium $\sigma_s$ varies only by $7\%$ about its mean, so the $\sigma_s$-weighted and standard $L^2$ inner products are nearly proportional, and the leading POD subspaces (hence the realized $L^2$ errors) are essentially identical (gain $\eps_{L^2}/\eps_{\sigma_s} \approx 1.00$ at every rank). The energy-norm optimality of Corollary~\ref{COR:ROM Energy} still holds exactly. It simply does not translate into a visible $L^2$ gain unless $\sigma_s$ is strongly heterogeneous and the regime is near-critical. We report the null
result rather than engineer a contrived high-contrast medium to inflate it. 

\subsection{ROM for anisotropic medium}
\label{SUBSEC:Num ANISO}

For anisotropic scattering, we use the Henyey-Greenstein phase function with anisotropy $g\in\{0,0.3,0.5,0.7,0.9\}$, spanning isotropic to strongly forward-peaked regimes. We verify that the anisotropic integral discretization of~\Cref{SUBSEC:Nystrom ANISO} reduces to the isotropic solve at $g=0$ (relative difference $4.5\times10^{-11}$). Two predictions of the analysis are borne out in~\Cref{FIG:ANISO g-Sweep}. First, anisotropy is an $\cO(\|p-1\|_\infty)=\cO(g)$ perturbation of the density: the measured $\|U_g-U_0\|/\|U_0\|$ grows essentially linearly in $g$ ($2.2\%,3.5\%,4.7\%,5.8\%$ at $g=0.3,0.5,0.7,0.9$, i.e. a slope of $\approx 0.06-0.07$ per unit $g$), confirming the prediction of Proposition~\ref{PROP:ANISO Pert}. Second, the density manifold stays low-dimensional: the energy rank is $r=8$ for $g\le0.7$ and only $r=9$ at $g=0.9$, so the reduced anisotropic basis is essentially the size of the isotropic one, confirming the prediction of Corollary~\ref{COR:ANISO n-Width}.
\begin{figure}[!htbp]
  \centering
  \includegraphics[width=0.8\linewidth]{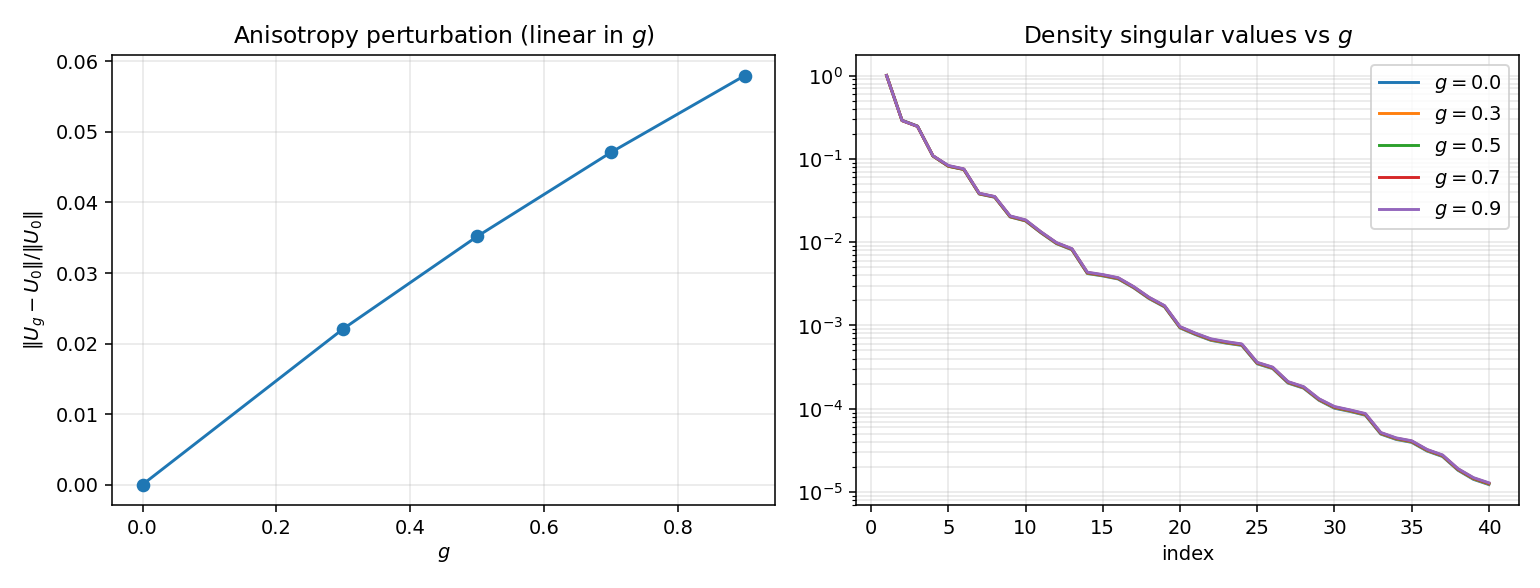}
  \caption{$\|U_g-U_0\|/\|U_0\|$ vs.\ $g$ (left) and the corresponding snapshot matrix singular values (right).}
  \label{FIG:ANISO g-Sweep}
\end{figure}

\begin{figure}[!htbp]
  \centering
  \includegraphics[width=0.9\linewidth]{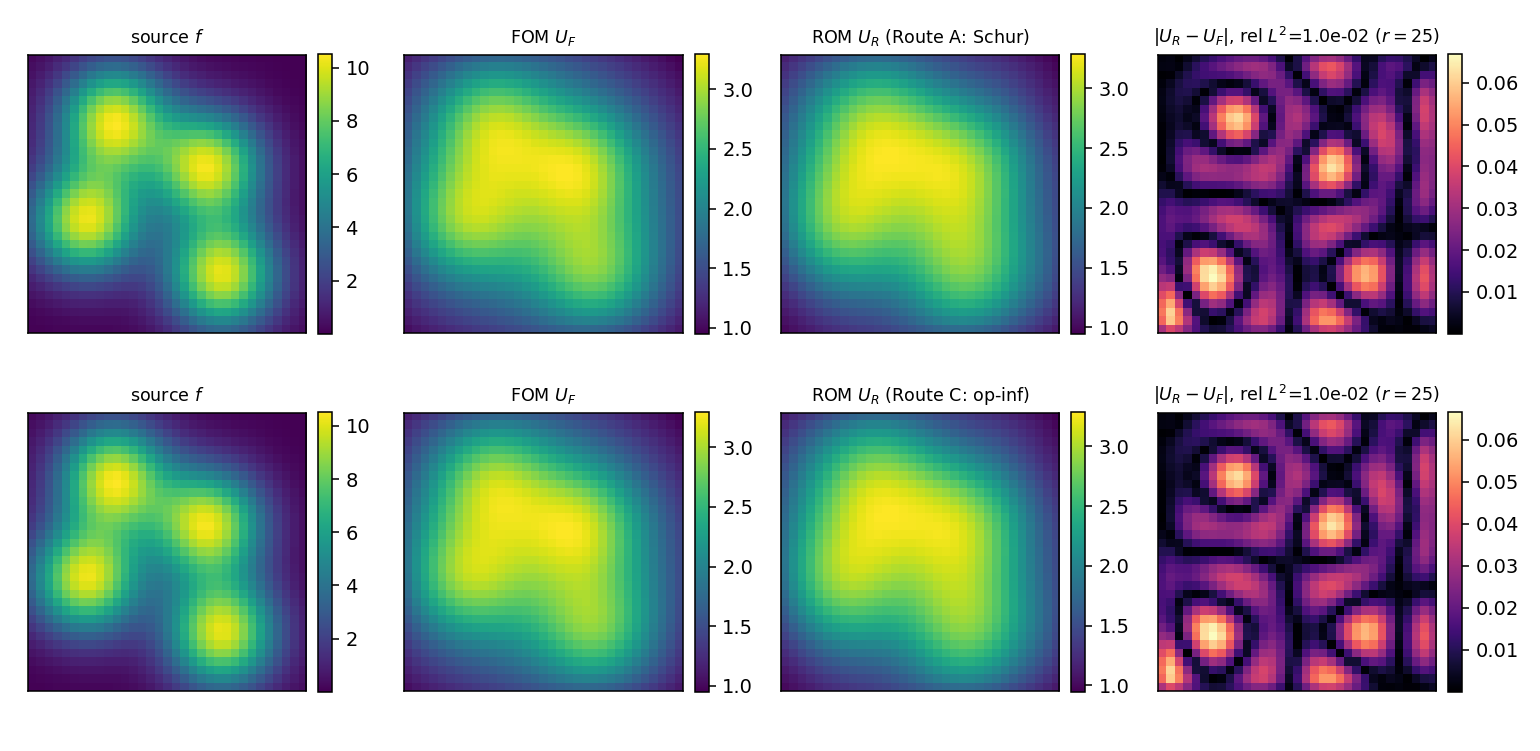}
  \caption{Comparison of FOM and ROM for an anisotropic medium with $g=0.7$ and $\gamma_0=0.8$. Shown from left to right are: source $f$, FOM density $U_F$, ROM density $U_R$, and pointwise error $|U_R-U_F|$. Top row: Route~A ROM; Bottom row: Route~B ROM.}
  \label{FIG:ANISO Reduction}
\end{figure}
\Cref{FIG:ANISO Reduction} is the anisotropic counterpart of the homogeneous~\Cref{FIG:Homo Reduction}. At $g=0.7$ and $\gamma_0=0.8$, we plot the reconstruction of the density for an unseen multi-Gaussian test source. Shown are the source function $f$, FOM density $U_F$, ROM density $U_R$, and the pointwise error $|U_R-U_F|$. Both the intrusive coupled-Schur ROM (Route~A, top row) and the data-driven operator inference ROM (Route~B, bottom row) achieve relatively good accuracy (with $L^2$ error $\sim 10^{-2}$ at $r=25$).

\begin{figure}[!htbp]
  \centering
  \includegraphics[width=0.8\linewidth]{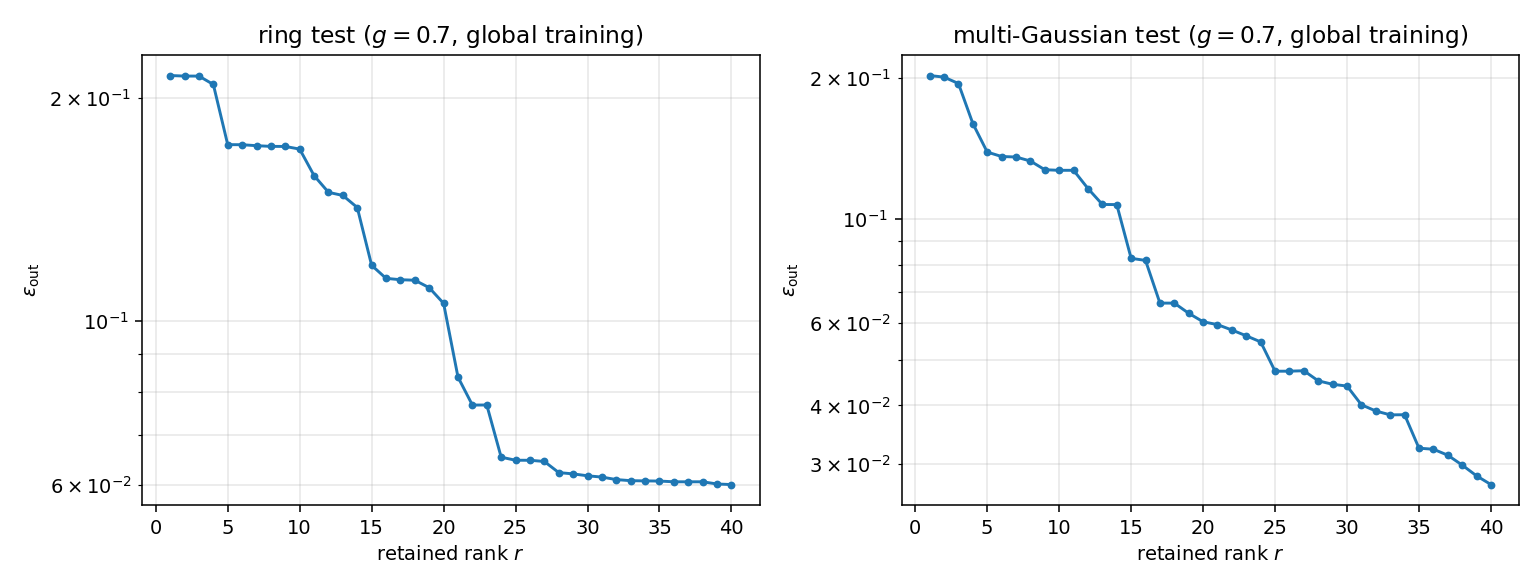}
  \caption{Out-of-sample error of the data-driven ROM (Route~B) vs.\ retained rank for global random-Fourier training, on a ring test (left) and a multi-Gaussian test (right).}
  \label{FIG:ANISO Error-Rank}
\end{figure}
\Cref{FIG:ANISO Error-Rank} is the anisotropic counterpart of the homogeneous
\Cref{FIG:Error-vs-Rank}. For fixed anisotropy $g=0.7$ and $\gamma_0=0.8$, we plot the out-of-sample testing error of the data-driven density ROM (Route~B, operator inference) as a function of retained rank. The training is performed using global random Fourier sources. As in the isotropic case, the error decreases slowly, again indicating that globally trained bases generalize poorly to localized test sources.

\begin{figure}[!htbp]
  \centering
  \includegraphics[width=0.55\linewidth]{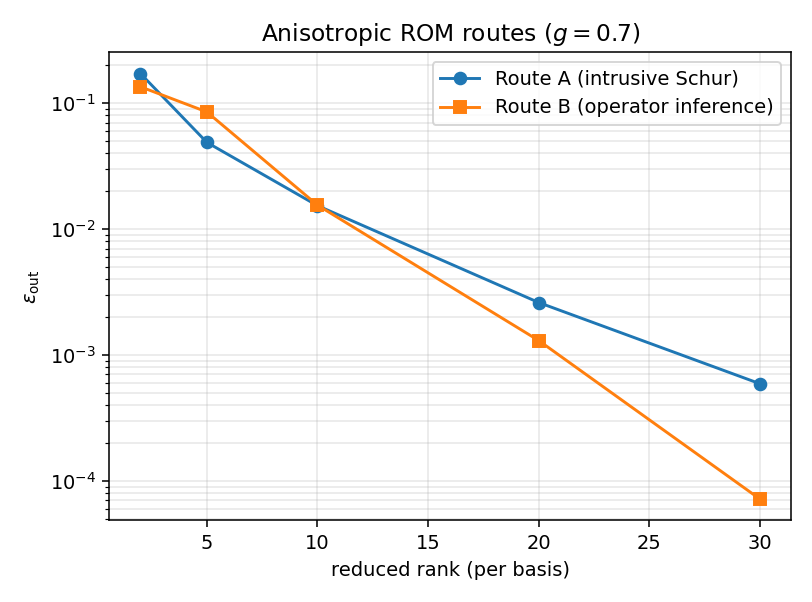}
  \caption{Out-of-sample error vs. reduced rank for the two anisotropic ROM routes at $g=0.7$.}
  \label{FIG:ANISO Routes}
\end{figure}
To compare the two anisotropic reductions of~\Cref{SUBSEC:ANISO ROM}, we show their error vs rank results in~\Cref{FIG:ANISO Routes}  at $g=0.7$. Route~A (intrusive joint reduction of the coupled Schur system with $M_a=16$) needs $r=10$ density modes and $r_\Phi=13$ fluctuation modes at $\eta=0.99$ and its out-of-sample error decays from $1.7\times10^{-1}$ at $r=2$ to $5.9\times10^{-4}$ at $r=30$. Route~B (non-intrusive operator inference from density snapshots with $M_a=24$) decays from $1.4\times10^{-1}$ to $7.2\times10^{-5}$. The adaptive Tikhonov regularization keeps the inferred reduced operator a contraction ($\|\widehat\bM\|_2<1$, the hypothesis of Proposition~\ref{PROP:Operator Inf}) at every rank. The two are comparable at low rank. At high rank, the data-driven Route~B is the most accurate here because it learns directly from the reference solver (no frozen-angle consistency error, no moment truncation), while Route A carries a certified error bound. This trade-off is exactly the one anticipated in~\Cref{SUBSEC:ANISO ROM}.

\subsection{Computational cost and speed-up } 

We quantify the speed-up of the ROMs. We focus on M2, the isotropic scattering ROM with Galerkin projection. We run simulations for a homogeneous medium in the intermediate regime of $\gamma_0=0.8$ ($\sigma_a=0.1$ with a constant $\sigma_s$). The basis is built offline from $N_s=400$ local Gaussian training sources~\eqref{EQ:Source Gaussian} of variance $\varsigma^2=0.05$, with rank $r=8$ selected by the energy criterion~\eqref{EQ:Energy} at $\eta=0.99$ (the same regime as~\Cref{TAB:Sigmas}). The full-order baseline $t_{\rm FOM}$ is a dense LU factor-and-solve of $(\bI-\bK)U=\phi$ (an independent factorization and solve per query with cost  $\cO(N_{\rm dof}^3)$), while the reduced query $t_{\rm ROM}$ forms the reduced load $\Rb^\top\boldsymbol\phi$ and solves the $r\times r$ system $(\bI_r-\widehat\bK)\bc=\Rb^\top\boldsymbol\phi$ with cost $\cO(r^3)$. We therefore expect that the gain grows like $N_{\rm dof}^3/r^3$ with problem size.
\begin{table}[!htbp]
\centering
\begin{tabular}{ccccccc}
\toprule
$N_x{=}N_y$ & $N_{\rm dof}$ & $r$ & $t_{\rm FOM}$ (ms) & $t_{\rm ROM}$ (ms)
    & speed-up & $\eps_{\rm out}$ \\
\midrule
33 & 1089 & 8 & 20.8    & 0.012 & $1.8\times10^{3}$ & 2.38e-02 \\
49 & 2401 & 8 & 201.0   & 0.060 & $3.4\times10^{3}$ & 2.38e-02 \\
65 & 4225 & 8 & 2794.3  & 0.121 & $2.3\times10^{4}$ & 2.38e-02 \\
97 & 9409 & 8 & 33610.7 & 0.264 & $1.3\times10^{5}$ & 2.38e-02 \\
\bottomrule
\end{tabular}
\caption{Scaling of per-query cost measured in wall time.}
\label{TAB:Time}
\end{table}

\Cref{TAB:Time} reports the per-query wall times across a grid ladder. All wall-clock times are the best of three repeats on a single machine. The out-of-sample accuracy $\eps_{\rm out}$ quoted is the relative $L^2$ error
on a single held-out Gaussian test source. The reduced solve is $1.8\times10^{3}$ times faster already at $N_{\rm dof}=1089$ and over $10^{5}$ times faster at $N_{\rm dof}=9409$. The one-time offline cost for $N_s=400$ shared-factorization snapshot solves, one SVD, and the projection $\widehat\bK=\Rb^\top\bK\Rb$ ranges from $0.3$\,s at $N_x=N_y=33$ to $25$\,s at $N_x=N_y=97$. The details are documented in~\Cref{TAB:Offline}.
\begin{table}[!htbp]
\centering
\begin{tabular}{l c c}
\toprule
$N_x{=}N_y$ ($N_{\rm dof}$) & total offline (s) & dominant stage \\
\midrule
33 (1089)  & 0.3  & $N_s$ snapshot solves \\
49 (2401)  & 0.9  & $N_s$ snapshot solves \\
65 (4225)  & 2.7  & assembly + SVD \\
97 (9409)  & 25.2 & dense assembly + factorization \\
\bottomrule
\end{tabular}
\caption{Summary of one-time offline cost measured in wall time.}
\label{TAB:Offline}
\end{table}

\section{Concluding remarks}
\label{SEC:Concl}

We developed a reduced-order modeling framework for the radiative transfer equation built on the integral formulation of the angularly averaged density. The averaging step provides smoothness and a genuinely low-dimensional solution manifold for the density, even outside the diffusive regime, while the integral formulation yields an exact, low-rank-friendly governing equation in physical space. On this foundation, we built projection-based ROMs with a contraction-based stability guarantee and a C\'ea-type error bound tied to the snapshot singular values, and for the anisotropic case, an intrusive joint reduction of the coupled system and a non-intrusive data-driven variant for black-box transport solves. Two-dimensional experiments confirm low-rank reducibility and substantial speed-ups across scattering regimes.

Many interesting directions for future investigation exist. One example is to generalize the current framework to the time-dependent RTE and multigroup RTEs. Time-dependent problems can probably be handled with the same type of method using either Fourier transforms or the integral formulation of~\cite{ZhZh-CSIAM20}. Fourier transform in time reduces the time-dependent RTE to our stationary form~\eqref{EQ:RTE} with a complex total attenuation coefficient. This does not seem to create extra obstacles for the ROM framework we developed to work. From an application point of view, generalizing the current framework to inverse transport problems with certified, residual-driven basis enrichment is of great interest.

\section*{Acknowledgments}

This work is partially supported by the National Science Foundation through grants DMS-2309802 (KR) and DMS-2529292 (ST). The work of KR is also partially supported by the Office of
Naval Research (ONR) award N000142612023.

{\small
\bibliography{BIB-ROM-ERT}
\bibliographystyle{siam}
}

\end{document}